\documentclass{amsproc}

\usepackage{ifpdf}
\newif\ifdvipdfmx
\dvipdfmxtrue
\ifpdf
  \usepackage{graphicx}
  \usepackage{xcolor}
  \definecolor{darkblue}{rgb}{0.05,0.10,0.45}
  \usepackage[
    colorlinks=true,
    linkcolor=darkblue,
    citecolor=darkblue,
    urlcolor=darkblue
  ]{hyperref}
\else
  \ifdvipdfmx
    \usepackage[dvipdfmx]{graphicx}
    \usepackage[dvipdfmx]{xcolor}
    \definecolor{darkblue}{rgb}{0.05,0.10,0.45}
    \usepackage[
      dvipdfmx,
      colorlinks=true,
      linkcolor=darkblue,
      citecolor=darkblue,
      urlcolor=darkblue
    ]{hyperref}
  \else
    \usepackage[dvips]{graphicx}
    \usepackage[dvips]{xcolor}
    \definecolor{darkblue}{rgb}{0.05,0.10,0.45}
    \usepackage[
      dvips,
      colorlinks=true,
      linkcolor=darkblue,
      citecolor=darkblue,
      urlcolor=darkblue
    ]{hyperref}
  \fi
\fi

\usepackage[all]{xy}
\usepackage{amssymb}
\usepackage{amscd}
\usepackage{braket}
\usepackage{enumitem}

\def\a{\alpha}
\def\b{\beta}
\def\c{\gamma}
\def\d{\delta}
\def\e{\epsilon}

\def\l{\lambda}
\def\s{\sigma}
\def\t{\tau}

\def\L{\Lambda}
\def\O{\Omega}

\def\AA{{\mathbb A}}

\def\NN{{\mathbb N}}
\def\PP{{\mathbb P}}

\def\ZZ{{\mathbb Z}}

\def\cal{\mathcal}

\def\cA{{\cal A}}

\def\cC{{\cal C}}

\def\cF{{\cal F}}

\def\cL{{\cal L}}

\def\cP{{\cal P}}

\def\cV{{\cal V}}

\def\Aut{\operatorname{Aut}}

\def\Coker{\operatorname{Coker}}

\def\CM{\operatorname {CM}}

\def\depth{\operatorname{depth}}

\def\det{\operatorname{det}}
\def\dim{\operatorname{dim}}

\def\End{\operatorname {End}}
\def\Ext{\operatorname {Ext}}

\def\GL{\operatorname {GL}}

\def\gldim{\operatorname{gldim}}

\def\grmod{\operatorname{grmod}}
\def\GrMod{\operatorname{GrMod}}

\def\GrAut{\operatorname{GrAut}}

\def\Hom{\operatorname {Hom}}

\def\id{\operatorname {id}}

\def\Im{\operatorname{Im}}

\def\ind{\operatorname {ind}}

\def\Ker{\operatorname {Ker}}

\def\mod{\operatorname{mod}}

\def\Obj{\operatorname{Obj}}

\def\Pic{\operatorname{Pic}}
\def\Proj{\operatorname{Proj}}

\def\tors{\operatorname{tors}}

\def\tails{\operatorname{tails}}

\def\uCM{\underline {\operatorname{CM}}}

\def\<{\langle}
\def\>{\rangle}

\def\NMF{\operatorname{NMF}}
\def\uNMF{\underline{\operatorname{NMF}}}
\def\LMF{\operatorname{LMF}}

\def\rnum#1{\expandafter{\romannumeral #1}}
\def\Rnum#1{\uppercase\expandafter{\romannumeral #1}}

\theoremstyle{plain}
\newtheorem{theorem}{Theorem}[section]
\newtheorem{corollary}[theorem]{Corollary}
\newtheorem{lemma}[theorem]{Lemma}
\newtheorem{proposition}[theorem]{Proposition}

\theoremstyle{definition}
\newtheorem{definition}[theorem]{Definition}
\newtheorem{example}[theorem]{Example}

\theoremstyle{remark}

\newtheorem{remark}[theorem]{Remark}

\numberwithin{equation}{section}
\renewcommand{\labelenumi}{(\arabic{enumi})}
\renewcommand{\labelenumii}{(\alph{enumii})}
\renewcommand{\theequation}{\thesection.\arabic{equation}}

\begin{document}

\title{Irreducible Noncommutative Quadrics}

\author{Izuru Mori}

\address{Department of Mathematics,
Faculty of Science,
Shizuoka University,
836 Ohya, Suruga-ku, Shizuoka 422-8529, Japan}
\email{mori.izuru@shizuoka.ac.jp}

\author{Kenta Ueyama}

\address{Department of Mathematics,
Faculty of Science,
Shinshu University,
3-1-1 Asahi, Matsumoto, Nagano 390-8621, Japan}
\email{ueyama@shinshu-u.ac.jp}

\author{Wenchao Wu}

\address{School of Mathematical Sciences,
University of Science and Technology of China,
Hefei Anhui 230026, China}
\email{wuwch20@mail.ustc.edu.cn}


\thanks{
The first author was supported by JSPS KAKENHI Grant Number JP25K06917.
The second author was supported by JSPS KAKENHI Grant Number JP26K06761.
}

\begin{abstract}
In this paper, we study irreducible noncommutative quadrics $S/(f)$ via noncommutative graded matrix factorizations. We show that the line modules over $S/(f)$ are described by the rulings arising from indecomposable noncommutative linear matrix factorizations of $f$ of rank $2$. We study when Zhang twists of a standard smooth irreducible noncommutative quadric are standard. Finally, by identifying all singular central Sklyanin quadrics, we prove that every smooth central Sklyanin quadric is standard.
\end{abstract}

\maketitle

\section{Introduction}
Throughout this paper, $k$ is an algebraically closed field of characteristic $0$.

Noncommutative quadric hypersurfaces (and their associated noncommutative projective schemes) are among the major objects of study in noncommutative algebraic geometry (see e.g.\ \cite{HU, HM, HMW26, IKU, MUnkp, SV, Uqh}).
The simplest non-trivial examples are noncommutative conics, which have been studied intensively in recent years; in particular, noncommutative
central conics were completely classified in \cite{HMM,HMW26}. The next
natural step is to study noncommutative quadrics.

In \cite{MUqps}, it was shown that a standard smooth noncommutative quadric $A=S/(f)$ which is a domain has a remarkable property: there exists a $3$-dimensional $\mathbb N$-graded AS-regular algebra $B$ such that $B_0$
is the path algebra $kK_2$ of the $2$-Kronecker quiver $K_2$ and
$\Proj_{nc} A\cong \Proj_{nc} B$. Hence the classification of such noncommutative quadrics is closely related to that of $3$-dimensional $\mathbb N$-graded AS-regular algebras, which are important objects in
noncommutative algebraic geometry and in representation theory of finite-dimensional algebras \cite{MM}.

On the other hand, it was also shown in \cite{MUqps} that there exists a non-standard smooth noncommutative quadric which is a domain. This raises the question of whether non-standard smooth irreducible noncommutative quadrics are common or exceptional among smooth irreducible noncommutative quadrics.

The noncommutative quadrics considered below are quotients $A=S/(f)$, where $S$ is a $4$-dimensional quantum polynomial algebra and $f\in S_2$ is a regular normal element; the precise assumptions used in this paper
will be stated in Section 3.
It is unclear to us whether one can always replace $f$ by a central element up to graded Morita equivalence, as indicated in \cite{SV}.
Moreover, it is often difficult to check whether $A$ is a domain, while it is often easier to check whether $f$ is irreducible.
For this reason, we work with normal elements $f$ which are not necessarily central, and assume that $f$ is irreducible rather than assuming that $A$ is a domain, although these two conditions may be equivalent in some situations.

Our main tool is the theory of noncommutative graded matrix factorizations developed in \cite{CCKM,MUnmf}. We use noncommutative graded matrix factorizations to study line modules over noncommutative quadrics. In particular, for a smooth irreducible noncommutative quadric, we give a clearer characterization of line modules in terms of the two rulings arising from noncommutative linear matrix factorizations of rank $2$ defined in this paper.
This provides a matrix-factorization interpretation of the rulings on a smooth noncommutative quadric considered in \cite{SV}.

The paper is organized as follows. In Section 2, we collect some definitions and preliminary results needed in this paper. In Section 3, we characterize line modules over noncommutative quadrics in terms of noncommutative graded matrix factorizations, and then prove some results that will be used later.

In Section 4, we study Zhang twists of a standard smooth noncommutative quadric and determine when the twists remain standard.
When $A=S/(f)$ is a standard smooth noncommutative quadric, $A$ has exactly
two isomorphism classes of indecomposable maximal Cohen-Macaulay modules generated in degree $0$ (or equivalently, of noncommutative linear matrix factorizations of rank $2$).
We denote them by $X$ and $Y$.
To $X$ and $Y$, respectively, we associate two families
${\mathsf L}_X$ and ${\mathsf L}_Y$ of lines on certain quadric surfaces in $\mathbb P^3$.
The following result gives a geometric criterion for
standardness of Zhang twists in terms of the induced action on these two families of lines.

\begin{theorem}[Theorem \ref{thm:switch}]
Let $A=S/(f)$ be a  standard smooth noncommutative quadric. For any $\s\in \GrAut S$ such that $\s(f)=f$ (so that $\s \in \GrAut A$),
$A^{\s}$ is standard if and only if $\s^*(\ell)\in {\mathsf L}_Y$ for every $\ell \in {\mathsf L}_Y$, and $A^{\s}$ is non-standard if and only if $\s^*(\ell)\in {\mathsf L}_X$ for every $\ell \in {\mathsf L}_Y$.
\end{theorem}

In the special case of the commutative smooth quadric, the two families ${\mathsf L}_X$ and ${\mathsf L}_Y$ recover the two classical rulings on the quadric surface.
Thus the above criterion can be expressed purely in
terms of whether the induced automorphism preserves or switches these rulings. This gives the following result.

\begin{theorem}[Theorem \ref{thm.cotw}]
Let $A=k[x, y, z, w]/(xw-yz)$ be the homogeneous coordinate ring of a smooth quadric $Q=\cV(xw-yz)\subset \PP^3$.
For $\s\in \GrAut A$,
$A^{\s}$ is standard if and only if $\s^*\in \Aut Q$ preserves the rulings of $Q$,
and $A^{\s}$ is non-standard if and only if $\s^*\in \Aut Q$ switches the rulings of $Q$.
\end{theorem}

Finally, in Section 5, we study central Sklyanin quadrics, that is, noncommutative quadrics $S/(f)$ where $S$ is a $4$-dimensional (non-degenerate) Sklyanin algebra and $0\neq f\in Z(S)_2$. We first identify the singular central Sklyanin quadrics explicitly and show that, for each fixed $S$, there is
only one singular central Sklyanin quadric up to isomorphism. We then use this classification to show that every smooth central Sklyanin quadric is standard.
More precisely, our result is as follows.

\begin{theorem}[Theorem \ref{thm_sing}]
Let $S$ be a $4$-dimensional  (non-degenerate) Sklyanin algebra and $0\neq f \in
k\Omega_1 + k\Omega_2 $, where $\Omega_1, \Omega_2\in Z(S)_2$ are certain central elements of degree $2$.
(In this case $f$ is irreducible by Proposition \ref{prop_cenirr}.)
\begin{enumerate}
\item The following conditions are equivalent.
\begin{enumerate}
\item $S/(f)$ is singular.
\item $f$ is one of $\Omega_1$, $\d_1(\Omega_1)$, $\d_2(\Omega_1)$,
or $\d_3(\Omega_1)$ up to scalar, where $\delta_1,\delta_2,\delta_3$ are certain graded algebra automorphisms of $S$.
\item $S/(f) \cong S/(\Omega_1)$.
\end{enumerate}
\item  If $S/(f)$ is smooth, then $S/(f)$ is standard.
\end{enumerate}
\end{theorem}

Thus the Sklyanin case provides a class of noncommutative
quadrics in which smoothness forces standardness, in contrast with the
existence of non-standard smooth noncommutative quadrics in general.

\section{Preliminaries}

\subsection{Terminology and Notation}

Throughout this paper, graded algebras are assumed to be $\mathbb Z$-graded.
A {\it connected graded algebra} is an $\mathbb N$-graded algebra $A=\bigoplus_{i\in\NN}A_i$ such that $A_0=k$.

For a graded algebra $A = \bigoplus_{i \in \ZZ} A_i$,
we denote by $\GrMod A$ the category whose objects are graded right $A$-modules and whose morphisms are degree-preserving right $A$-module homomorphisms.
We denote by $\grmod A$ the full subcategory of $\GrMod A$ consisting of finitely generated modules.

 We say that $M\in \GrMod A$ is {\it locally finite} if $\dim _kM_i<\infty$ for all $i\in \ZZ$, and in this case, we define the {\it Hilbert series} of $M$ by
$
H_M(t) : = \sum_{i \in \mathbb{Z}} (\dim_{k} M_i) t^i \in \mathbb{Z}[[t,t^{-1}]].
$

For $M \in \GrMod A$ and $j\in \ZZ$, we define the {\it shift} $M(j) \in \GrMod A$ by $M(j)_i : = M_{j+i}$ for $i \in \mathbb{Z}$.
With a slight abuse of notation, we set
\[
\Ext^i_A (M,N)_0: = \Ext^i_{\GrMod A} (M, N)\ \
\text{and}\ \
\Ext^i_A (M,N):= \bigoplus_{j \in \mathbb{Z}} \Ext^i_A (M, N(j))_0
\]
for $M, N \in \GrMod A$.

Let $\cC$ be an additive category and $\cP$ a set of objects of $\cC$ closed under finite direct sums.
Then the \emph{factor category} $\cC/\cP$ has $\Obj(\cC/\cP) =\Obj(\cC)$ and $\Hom_{\cC/\cP}(M, N) = \Hom_{\cC}(M, N)/\cP(M,N)$ for $M, N\in \Obj(\cC/\cP)=\Obj(\cC)$,
where $\cP(M,N)$ is the subgroup consisting of all morphisms from $M$ to $N$ that factor through objects in $\cP$. Note that $\cC/\cP$ is also an additive category.

Let $A$ be a connected graded algebra and let $M\in \GrMod A$ be a bounded-below graded $A$-module. Since a minimal graded free resolution of $M$
\[
\xymatrix@R=0.5pc@C=3pc{
	\cdots \ar[r] &F^2 \ar[r]^{\phi^1} &F^1 \ar[r]^{\phi^0} &F^0 \ar[r]^{\e_M} &M \ar[r] &0
}
\]
is unique up to isomorphism, we may define $\Omega M:=\Ker \epsilon_M$ up to isomorphism.

We recall a nice operation on graded algebras, called {\it Zhang twist}, introduced by Zhang \cite{Zh}.
Let $A$ be a graded algebra and $\s\in \GrAut A$ a graded algebra automorphism. The \emph{Zhang twist} of $A$ by $\s$ is a graded algebra $A^\s$ where $A^\s=A$ as a graded $k$-vector space with the new multiplication
\[
a^\s b^\s=(a\s^i(b))^\s
\]
for $a^\s\in A^\s_i$ and $b^\s\in A^\s$.
Here we write $a^\s\in A^\s$ for $a\in A$ when viewed as an element of $A^\s$, and the product $a\s^i(b)$ is computed in $A$.

Let $M\in \GrMod A$. We define a graded right $A^\s$-module $M^\s$ as follows. As a graded $k$-vector space, $M^\s=M$.
For $m\in M_i$ and $a\in A$, we define the right $A^\s$-action by $m^\s a^\s := (m\s^i(a))^\s$.
Here $m^\s$ denotes the element $m\in M$ viewed as an element of $M^\s$, and the product $m\s^i(a)$ is computed in $M$.
If $\varphi:M\to N$ is a morphism in $\GrMod A$, then the same graded $k$-linear map defines a morphism $\varphi^\s:M^\s\to N^\s$ in $\GrMod A^\s$. By \cite[Theorem 3.1]{Zh}, this construction gives an equivalence
\begin{equation}\label{eZh}
(-)^\s:\GrMod A\longrightarrow \GrMod A^\s,
\quad
M\longmapsto M^\s .
\end{equation}

\subsection{Noncommutative Graded Matrix Factorizations}
In this subsection, we recall some background results on  noncommutative graded matrix factorizations obtained in \cite{MUnmf}.

\begin{definition}[{\cite[Definition 2.1]{MUnmf}}]
	Let $S$ be a graded ring and $f\in S_d$ a homogeneous element.
	A \emph{noncommutative graded (right) matrix factorization} of $f$ over $S$ of rank $r$ is a sequence of graded right $S$-module homomorphisms $\phi: = \{\phi^i:F^{i+1}\to F^i\}_{i\in \ZZ}$ , where $F^i= \bigoplus _{s=1}^rS(-m_{is})$ for some $m_{is} \in \ZZ$ for every $i\in \ZZ$, and $\phi^i\phi^{i+1}=f\cdot$.

	A morphism $\mu :\{\phi^i:F^{i+1}\to F^i\}_{i\in \ZZ}\to \{\psi^i:G^{i+1}\to G^i\}_{i\in \ZZ}$
	of noncommutative graded (right) matrix factorizations is a sequence of right $S$-module homomorphisms $\{\mu ^i:F^i\to G^i\}_{i\in \ZZ}$ such that the diagram
	\[\xymatrix@R=2pc@C=3pc{
		F^{i+1} \ar[d]_{\mu ^{i+1}} \ar[r]^{\phi^i} &F^i \ar[d]^{\mu ^{i}} \\
		G^{i+1} \ar[r]^{\psi^i} &G^{i}
	}\]
	commutes for every $i\in \ZZ$.
	The category of noncommutative graded (right) matrix factorizations is denoted by $\NMF^{\ZZ}_S(f)$.
\end{definition}

For a noncommutative graded matrix factorization $\phi=\{\phi^i:F^{i+1}\to F^i\}_{i\in \ZZ}$, we define the {\it degree shift} by $\phi(1):=\{(\phi(1))^i=\phi^{i}(1):F^{i+1}(1)\to F^{i}(1)\}_{i\in \ZZ}$, and the {\it position shift} by $\phi[1]:=\{(\phi[1])^i=\phi^{i+1}:F^{i+2}\to F^{i+1}\}_{i\in \ZZ}$.
The shifts $(1)$ and $[1]$ define autoequivalences of $\NMF^{\ZZ}_S(f)$.

\begin{remark} \label{rem.lambda}
\begin{enumerate}
\item
Let $\{\phi^i:F^{i+1}\to F^i\}_{i\in \ZZ}$ be a noncommutative graded
matrix factorization of $f$ over $S$ of rank $r$, where
$F^i=\bigoplus_{s=1}^r S(-m_{is})$.
The homomorphism $\phi^i$ is represented by the left multiplication by a matrix $\Phi^i$ whose entries are homogeneous elements of $S$, and $\Phi^i\Phi^{i+1}=fE_r$, where $E_r$ is the identity matrix of size $r$.
\item
By (1), we often write $\Phi^i$ for $\phi^i$, and
a noncommutative graded matrix factorization as $\{\Phi^i\}_{i\in \ZZ}\in \NMF^{\ZZ}_{S}(f)$ by abuse of notation.
If $f$ is regular normal, then
it follows from \cite[Theorem 4.4]{MUnmf} that
$\{\Phi^i\}_{i\in \ZZ}$ is uniquely determined by $\Phi^0$ and $\Phi^1$,
so we also simply write $(\Phi^0, \Phi^1)\in \NMF_S^{\ZZ}(f)$.
\end{enumerate}

\end{remark}

\begin{definition}[{\cite[Definition 6.3]{MUnmf}}]
	Let $S$ be a graded ring and $f\in S_d$. For a graded free module $F \in \grmod S$,
	we define $\phi_F, {}_F\phi \in \NMF^{\mathbb Z}_S(f)$
	by
	\[
	\phi_F^{2i}=\operatorname{id}_F:F(-id)\longrightarrow F(-id),
	\qquad
	\phi_F^{2i+1}=f\cdot :F(-id-d)\longrightarrow F(-id),
	\]
	\[
	{}_F\phi^{2i}=f\cdot :F(-id-d)\longrightarrow F(-id),
	\qquad
	{}_F\phi^{2i+1}=\operatorname{id}_F :F(-id-d)\longrightarrow F(-id-d).
	\]
	We define
\begin{align*}
&\mathcal F:=\{\phi_F\mid
\text{$F$ is a graded free module in $\grmod S$}
\},\\
&\mathcal G:=\{\phi_F \oplus {}_G\phi\mid
\text{$F$ and $G$ are graded free modules in $\grmod S$}
\}
\end{align*}
	and $\underline{\NMF}_S^\mathbb{Z}(f): = \NMF_S^\mathbb{Z}(f)/\mathcal{G}$.
\end{definition}

For $\phi\in \NMF_S^{\ZZ}(f)$, we define $\Coker \phi:=\Coker \overline {\Phi^0}\in \grmod S/(f)$.

\begin{lemma}[{\cite[Proposition 6.4]{MUnmf}}] \label{lem.conm}
Let $S$ be a graded algebra and $f\in S$ a homogeneous regular normal element.	Then the functor  $\Coker : \NMF_S^{\ZZ}(f)/\cF\to \grmod S/(f)$ is fully faithful.
\end{lemma}

Furthermore, we define
\[
\NMF^{0}_{S}(f):= \left\{ \{\phi^i:F^{i+1}\to F^i\}_{i\in\ZZ}\in \NMF^{\ZZ}_{S}(f)\ \middle|\ F^0 \text{ is generated in degree } 0 \right\}.
\]

\begin{definition} A noncommutative graded matrix factorization $\phi\in \NMF^0_S(f)$ is called a {\it noncommutative linear matrix factorization} if all the entries of $\Phi^i$ are in $S_1$.
The full subcategory of $\NMF^0_S(f)$ consisting of noncommutative linear matrix factorizations (of rank $r$) is denoted by $\LMF_S(f)$ (by $\LMF^r_S(f)$).
\end{definition}

	Let $S$ be a graded algebra.  For $\Phi, \Psi\in M_r(S)$, we define $\Phi\sim \Psi$ if there are $P, Q\in \GL_r(k)$ such that $\Psi=P\Phi Q$.

	\begin{lemma} \label{lem.pp3}  Let $S$ be a graded algebra and $f\in S$ a homogeneous regular normal element.	Then the natural functor $\LMF_S(f)\to \NMF_S^{\ZZ}(f)/\cF$ is fully faithful.  Moreover, for $\phi, \psi\in \LMF_S^r(f)$, the following are equivalent:
		\begin{enumerate}
			\item{} $\phi\cong \psi$.
			\item{} $\Coker \phi\cong \Coker \psi$.
			\item{} $\Phi^0\sim \Psi^0$.
		\end{enumerate}
	\end{lemma}

	\begin{proof} For $\phi, \psi\in \LMF_S(f)$, if $\mu\in\cF(\phi, \psi)$,
		then $\mu:\phi\to \psi$ factors as
\[
\xymatrix@C=2.1cm@R=0.7cm{
S(-1)^r \ar[r]^-{\phi^0=\Phi^0\cdot} \ar[d]_-{\alpha^1}
  & S^r \ar[d]^-{\alpha^0} \\
\bigoplus_i S(\ell_i) \ar[r]^-{\id} \ar[d]_-{\beta^1}
  & \bigoplus_i S(\ell_i) \ar[d]^-{\beta^0} \\
S(-1)^s \ar[r]^-{\psi^0=\Psi^0\cdot}
  & S^s
}
\]
		Since $\b^1\a^0=0:S^r\to S(-1)^s$, we have $\mu^0=\b^0\a^0=\psi^0\b^1\a^0=0$ and $\mu^1=\b^1\a^1=\b^1\a^0\phi^0=0$, so
		we can conclude that $\mu=0$.  This implies that
		$$\Hom_{\NMF_S^{\ZZ}(f)/\cF}(\phi, \psi)=\Hom_{\NMF_S^{\ZZ}(f)}(\phi, \psi)/\cF(\phi, \psi)=\Hom_{\NMF_S^{\ZZ}(f)}(\phi, \psi),$$
		so the natural functor $\LMF_S(f)\to \NMF_S^{\ZZ}(f)/\cF$ is fully faithful.

		It follows that $\phi\cong \psi$ in $\LMF_S(f)$ if and only if  $\phi\cong \psi$ in $\NMF_S^{\ZZ}(f)/\cF$ if and only if $\Coker \phi\cong \Coker \psi$ in $\grmod S/(f)$ by Lemma \ref{lem.conm}, so we obtain (1) $\Leftrightarrow$  (2).
		(2) $\Leftrightarrow$  (3) is clear.
	\end{proof}

\section{Line Modules and Rulings for Noncommutative Quadrics}

In this section, we study line modules over noncommutative quadrics via
noncommutative graded matrix factorizations. We then define rulings on
irreducible noncommutative quadrics and give a criterion, in terms of exact
sequences involving line modules and their shifts, for a smooth irreducible
noncommutative quadric to be standard or non-standard.

\subsection{Noncommutative Quadrics}

A quantum polynomial algebra, as defined below, is a noncommutative analogue of a commutative polynomial algebra in noncommutative algebraic geometry.
\begin{definition} \label{def.qpa}
	A noetherian connected graded algebra $S$ is called an $n$-dimensional {\it quantum polynomial algebra} if
	\begin{enumerate}
		\item{} $\gldim S=n$,
		\item{} $\Ext^i_S(k, S)=\begin{cases}  k(n) & \textnormal {if } i=n, \\
			0 & \textnormal {if } i\neq n, \end{cases}$
		\item{} $H_S(t)=(1-t)^{-n}$.
	\end{enumerate}
\end{definition}

By \cite{ATV2}, every 4-dimensional quantum polynomial algebra is a domain.
We will additionally assume
the following condition (*) for a 4-dimensional quantum polynomial algebra $S$, namely,
\begin{enumerate}
\item{} $S$ is Auslander-regular, and
\item{} $S$ satisfies the Cohen-Macaulay property
\end{enumerate}
in the sense of \cite{LS} to use the results from \cite{LS}.

\begin{remark}
As far as the authors know, no example is known of a 4-dimensional quantum polynomial algebra which does not satisfy condition (*).
\end{remark}

Since the noncommutative projective scheme associated to a 4-dimensional quantum polynomial algebra is regarded as a noncommutative analogue of $\PP^3$, we will make the following definition.

\begin{definition}
We say that $A=S/(f)$ is (the homogeneous coordinate ring of) a {\it noncommutative quadric} if $S$ is a 4-dimensional quantum polynomial algebra satisfying (*) and $f\in S_2$ is a (regular) normal element.

We say that a noncommutative quadric $A$ is {\it irreducible} if
$f$ is irreducible. Otherwise, $A$ is called {\it reducible}.

We say that a noncommutative quadric $A$ is {\it smooth} if the Serre quotient category $\tails A:= \grmod A/\tors A$ has finite global dimension, where $\tors A$ is the full subcategory consisting of finite-dimensional modules. Otherwise, $A$ is called {\it singular}.
\end{definition}

For a noncommutative quadric $A$,
we  define
\begin{align*}
\CM^{\ZZ}(A)&:=\{M\in \grmod A\mid \Ext_A^i(M, A)=0 \; \; \forall i\geq 1\}, \\
\CM^{0}(A)&:=\{M\in \CM^{\ZZ}(A)\mid M=M_0A\}, \\
\uCM^{\ZZ}(A)&:=\CM^{\ZZ}(A)/\cP \quad \textnormal{where $\cP$ is the set of all graded free right $A$-modules},\\
\mathbb M&:=\{M\in \ind(\CM^0(A))\mid M_0\cong k^2\}/\cong.
\end{align*}

If $A=S/(f)$ is a noncommutative quadric,  then there exists a unique regular normal element $f^!\in A_2^!$ up to scalar such that $S^! = A^!/(f^!)$ by \cite[Corollary]{ST}. We define $C(A): = A^![(f^!)^{-1}]_0$, which plays an essential role to study $A$ \cite{SV}.

\begin{lemma}  \label{lem.pp4}
Let $A=S/(f)$ be a noncommutative quadric.
\begin{enumerate}
\item{} \textnormal{(\cite[Theorems 6.5 and 6.6]{MUnmf})} There are equivalences
$\NMF^{\ZZ}_S(f)/\cF\cong \CM^{\ZZ}(A)$ and
$\uNMF^{\ZZ}_S(f)\cong \uCM^{\ZZ}(A)$.
\item{} \textnormal{(\cite[Lemma 4.13]{MUnkp})} There is an equivalence $\uCM^{0}(A) \cong \mod C(A)$.
\end{enumerate}
\end{lemma}

\begin{proposition} \label{prop.sism}
Let $A=S/(f)$ be a noncommutative quadric.
\begin{enumerate}
\item If $A$ is smooth irreducible, then $C(A)\cong M_2(k)\times M_2(k)$, so ${\mathbb M}$ consists of two objects.
\item If $A$ is singular irreducible, then $C(A)\cong M_2(k[x]/(x^2))$, so ${\mathbb M}$ consists of one object.
\end{enumerate}
\end{proposition}

\begin{proof}
Let $C(A) = \bigoplus^{s}_{i=1} P_i^{m_i}$ be an indecomposable decomposition of $C(A)$
($P_i \ncong P_j$ for $i \neq j$).
Then every simple module is given by $S_i = P_i/P_i J_{C(A)}$ for $i= 1, \dots, s$, where $J_{C(A)}$ is the Jacobson radical of $C(A)$.
Since $k$ is algebraically closed,
we have
\begin{align*}
\dim_k C(A) &= \sum_{i=1}^s m_i(\dim_k P_i) = \sum_{i=1}^s m_i(\dim_k S_i + \dim_k P_i J_{C(A)})\\
&=\sum_{i=1}^s m_i^2  + \sum_{i=1}^s m_i\dim_k P_i J_{C(A)}.
\end{align*}
Since $f$ is irreducible, $m_i=\dim_k S_i\geq 2$ by \cite[Lemma 5.11]{MUnkp}.  Since  $\dim_k C(A)=8$ by \cite[Lemma 4.13 (1)]{MUnkp},
one of the following two cases occurs:

(a) $s=2$ and $m_1 = m_2 = 2$;
In this case, $P_i J_{C(A)}=0$ for any $i$, so $J_{C(A)}=0$.
Thus $C(A)$ is semisimple. Since $k$ is algebraically closed, $C(A) \cong M_2(k) \times M_2(k)$.

(b) $s=1$ and $m_1 = 2$;
In this case, $C(A) = P_1^2$ in $\mod C(A)$.
Let $\Lambda = \End_{C(A)}(P_1)$. Then $\Lambda$ is a local ring and $ C(A) \cong \End_{C(A)}(P_1^2) \cong M_2(\Lambda)$.
Since $\dim_k \Lambda = (\dim_k C(A))/4 = 2$, we have $\Lambda \cong k[x]/(x^2)$, so $C(A) \cong M_2(k[x]/(x^2))$.

Since $A$ is smooth if and only if $C(A)$ is semisimple  by \cite[Theorem 5.5]{MUnkp},
the result follows from Lemma \ref{lem.pp4} (2).
\end{proof}

\begin{lemma} \label{lem.mphi}
Let $A=S/(f)$ be a noncommutative quadric.
\begin{enumerate}
\item{} $X\in \mathbb M$ if and only if there exists an indecomposable
$\phi\in \LMF^2_S(f)$ unique up to isomorphism such that $X\cong \Coker \phi$.
\item For $X\in \mathbb M$, we have $\Omega ^iX(i) \in \mathbb M$ for every $i\in \ZZ$.
\item If $A$ is irreducible, then
$\Omega ^2X(2)\cong X$ for $X\in \mathbb M$.
\end{enumerate}
\end{lemma}

\begin{proof}  (1) If $X\in \mathbb{M}\subset \CM^0(A)$, then $X$ has no free summand, so $X\cong \Coker \phi$ for some $\phi\in \LMF_{S}^2(f)$ by the proof of \cite[Proposition 7.8]{MUnmf} (1).
By Lemmas \ref{lem.pp3} and  \ref{lem.pp4} (1), the functor $\LMF_S^2(f)\to \CM^0(A)$ is fully faithful, so $\phi$ is indecomposable. The uniqueness follows from Lemma \ref{lem.pp3}.  The converse is clear.

(2) Since $
\phi\in \LMF^2_S(f)$ is an (indecomposable) noncommutative linear matrix factorization if and only if
so is $\phi[1](1)\in \LMF^2_S(f)$,
it follows from (1) that
$X=\Coker \phi
\in \mathbb M$ if and only if $\Omega ^iX(i)\cong \Coker (\phi[i](i))
\in \mathbb M$ for every $i\in \ZZ$.

(3) If  $\mathbb M=\{X\}$, then $\Omega^2 X(2) \cong X$ by (2). If  $\mathbb M=\{X, Y\}$, then either $\Omega X(1)\cong X, \Omega Y(1)\cong Y$ or $\Omega X(1)\cong Y, \Omega Y(1)\cong X$ by (2),
so $\Omega^2 X(2) \cong X$ in either case.
\end{proof}

\begin{definition}[{\cite[Section 5]{MUqps}}]
Let $A$ be a smooth irreducible noncommutative quadric and $\mathbb M=\{X, Y\}$.
We say that $A$ is {\it standard} if $\Omega X(1)\cong Y, \Omega Y(1)\cong X$, and {\it non-standard} if $\Omega X(1)\cong X, \Omega Y(1)\cong Y$.
\end{definition}

\begin{remark} \cite[Corollary 5.7]{SV} claims that every smooth noncommutative quadric $A=S/(f)$ is standard if $A$ is a domain and $f\in Z(S)_2$.  We will see later that this is not always the case (see Example \ref{ex.coex}).
One of the motivations of this paper is to show that this claim is generically true (see Theorem \ref{thm.Skl}).
\end{remark}

\subsection{Line Modules over Noncommutative Quadrics}

In noncommutative algebraic geometry, line modules defined below play an important role.

\begin{definition}
Let $A$ be a graded algebra finitely generated in degree 1.  We say that $L\in \GrMod A$ is a {\it line module} over $A$ if it is cyclic and $H_L(t)=(1-t)^{-2}$.
\end{definition}

\begin{lemma}
\label{lem.LS}
Let $S$ be a 4-dimensional quantum polynomial algebra satisfying the condition (*).
Then $L$ is a line module over $S$ if and only if $L\cong S/uS+vS$ for some linearly independent elements $u, v\in S_1$ such that $uS_1\cap vS_1\neq 0$.
\end{lemma}

\begin{proof}  In \cite[Proposition 2.8]{LS}, this was proved for a graded algebra $S$ of finite global dimension containing a regular normal sequence $\Omega_1, \Omega_2\in S_2$ such that $S/(\Omega_1, \Omega_2)\cong B(E, \s, \cL)$ is a twisted homogeneous coordinate ring with $\deg \cL=4$ (see the beginning of \cite[Section 2]{LS}), however, the proof only requires that $H_S(t)=(1-t)^{-4}$ and $S/uS$ is a 3-critical module for every $0\neq u\in S_1$, which hold if $S$ is a 4-dimensional quantum polynomial algebra satisfying the condition (*)
by \cite[Corollary 1.11]{LS}.
\end{proof}

\begin{lemma} \label{lem.lphi}
Let $A=S/(f)$ be a noncommutative quadric.
For every line module $L$, there exists
$\phi\in \LMF^2_S(f)$ and $(a, b)\in \PP^1$ such that
$L\cong \Coker ((a, b)\overline {\Phi^0})$.
\end{lemma}

\begin{proof} By \cite[Proposition 7.2 (2)]{LS}, every line module $L$ has a linear free resolution of the form
$$\begin{CD} \cdots @>>>A(-3)^2 @>>>A(-2)^2 @>>> A(-1)^2 @>\a>> A @>\b>> L\to 0,\end{CD}$$
so, for $X:=(\Ker \b)(1)$,
there exists
$\phi\in \LMF^2_S(f)$ such that $X\cong \Coker (\phi[1](1))\in \CM^0(A)$.
Since $Y:=\Omega ^{-1}X(-1)\cong \Coker \phi \in \CM^0(A)$, an exact sequence
$$0\to X(-1)\to A^2\to Y\to 0$$
induces an exact sequence
$$0\to \Hom_A(Y, A)_0\to\Hom_A(A^2, A)_0\to \Hom_A(X(-1), A)_0\to \Ext^1_A(Y, A)_0=0,$$
so
$$\begin{CD}\a: A(-1)^2 @>\overline{\Phi^0}\cdot>> A^2 @>(a, b)\cdot>> A \end{CD}$$
for some $(a, b)\in \PP^1$, hence $L\cong \Coker \a=\Coker ((a, b)\overline {\Phi^0})$.
\end{proof}

\begin{lemma} \label{lem.lcor}
Let $A=S/(f)$ be a noncommutative quadric, and   $\phi\in \LMF^2_S(f)$.  For $(a, b)\in \PP^1$,
$$L=(\Coker (a, b)\overline {\Phi^0})\cong A/(a\phi_1+b\phi_3)A+(a\phi_2+b\phi_4)A,$$
where $\Phi^0=\begin{pmatrix} \phi_1 & \phi_2 \\ \phi_3 & \phi_4 \end{pmatrix}$, is a line module if and only if $a\phi_1+b\phi_3, a\phi_2+b\phi_4\in S_1$ are linearly independent and $(a\phi_1+b\phi_3)S_1\cap (a\phi_2+b\phi_4)S_1\neq 0$.
\end{lemma}

\begin{proof}
For $(a, b)\neq (0, 0)$, there exists $(c, d)\neq (0, 0)$ such that $P=\begin{pmatrix} a & b \\ c & d \end{pmatrix}\in \GL_2(k)$.  Since $\{\Phi^i\}_{i\in \ZZ}\in \LMF^2_S(f)$ if and only if $\{\Psi^i\}_{i\in \ZZ}\in \LMF^2_S(f)$ where $\Psi^{2j}=P\Phi^{2j}, \Psi^{2j+1}=\Phi^{2j+1}P^{-1}$ for $j\in \ZZ$,
we may assume that $(a, b)=(1, 0)$.

Let $\phi\in \LMF^2_S(f)$.
If
$\Phi^{1}=\begin{pmatrix} \psi_1 & \psi_2 \\ \psi_3 & \psi_4 \end{pmatrix}$, then
$$\begin{pmatrix} f & 0 \\ 0 & f \end{pmatrix} = \Phi^0\Phi^{1}=\begin{pmatrix} \phi_1\psi_1+\phi_2\psi_3 & \phi_1\psi_2+\phi_2\psi _4 \\ \phi_3\psi_1+\phi_4\psi_3 & \phi_3\psi_2+\phi_4\psi_4 \end{pmatrix}$$
in $S$,
so $f=\phi_1\psi_1+\phi_2\psi_3\in \phi_1S_1+\phi_2S_1$.  Since $f\in S_2$ is normal,  $(f)=fS\subset \phi_1S+\phi_2S$, so  $A/\phi_1A+\phi_2A\cong S/\phi_1S+\phi_2S$.
By Lemma \ref{lem.LS}, $S/\phi_1S+\phi_2S$ is a line module over $S$ or equivalently a line module over $A$
if and only if $\phi_1, \phi_2\in S_1$ are linearly independent and $\phi_1S_1\cap \phi_2S_1\neq 0$.
 \end{proof}

\begin{lemma} \label{lem.irr}  Let $\phi\in \LMF^2_S(f)$ and $\Phi^0=\begin{pmatrix} \phi_1 & \phi_2 \\ \phi_3 & \phi_4 \end{pmatrix}$.
\begin{enumerate}
\item{} If $\Phi^0\sim \begin{pmatrix} \phi'_1 & \phi'_2 \\ \phi'_3 & \phi'_4 \end{pmatrix}$ such that $\phi'_j=0$ for some $j=1, 2, 3, 4$,
then $f$ is reducible.
\item{} If $f$ is irreducible, then
$$\{\phi_1, \phi_2\}, \{\phi_3, \phi_4\}, \{\phi_1, \phi_3\}, \{\phi_2, \phi_4\}$$
are linearly independent pairs.
\end{enumerate}
\end{lemma}

\begin{proof}

Although the proof is straightforward, we include it for completeness.

(1) If there exists $P, Q\in \GL_2(k)$ such that $P\Phi^0Q=\begin{pmatrix} \phi_1' & 0 \\ \phi_3' & \phi_4' \end{pmatrix}$, then
$$\begin{pmatrix} f & 0 \\ 0 & f \end{pmatrix}=\Phi^0\Phi^{1}=P\Phi^0QQ^{-1}\Phi^{1}P^{-1}=\begin{pmatrix} \phi_1'\psi'_1 & * \\ * & * \end{pmatrix},$$
where $Q^{-1}\Phi^{1}P^{-1}=\begin{pmatrix} \psi_1' & \psi_2' \\ \psi_3' & \psi_4' \end{pmatrix}$, so $f$ is reducible.  The other cases are similar.

(2) If $\phi_1, \phi_2\in S_1$ are linearly dependent, then there exists $P\in \GL_2(k)$ such that $\Phi^0P=\begin{pmatrix} \phi_1' & 0 \\ \phi_3' & \phi_4' \end{pmatrix}$,
so
$f$ is reducible by (1).
The other cases are similar.
\end{proof}

\begin{lemma} \label{lem.cocoW}
Let $A=S/(f)$ be an irreducible noncommutative quadric.  Then $L$ is a line module over $A$ if and only if
$$L\cong \Coker ((a, b)\overline {\Phi^0})$$
for some (indecomposable)
$\phi\in \LMF^2_S(f)$
and for some $(a, b)\in \PP^1$.
\end{lemma}

\begin{proof}
Without loss of generality, we may assume that $(a, b)=(1, 0)$.
By Lemmas \ref{lem.lphi} and \ref{lem.lcor}, it is enough to show that $\phi_1, \phi_2\in S_1$ are linearly independent and $\phi_1S_1\cap \phi_2S_1\neq 0$ where $\Phi^0=\begin{pmatrix} \phi_1 & \phi_2 \\ \phi_3 & \phi_4 \end{pmatrix}$.  By Lemma \ref{lem.irr}, $\phi_1, \phi_2$ are linearly independent.  Since $S$ is a domain and all $\phi_i,\psi_i$ are nonzero, we have $0\neq \phi_1\psi_2=-\phi_2\psi_4\in \phi_1S_1\cap \phi_2S_1$, so the result follows from Lemma \ref{lem.lcor}.
\end{proof}

\begin{remark} \label{rem.cocoW}
By the argument in the above lemmas,  it may be more reasonable to claim that $A/uA+vA$ is a line module over $A$ if and only if there exists $\phi\in \LMF^2_S(f)$ such that $\Phi^0=\begin{pmatrix} u & v \\ * & * \end{pmatrix}$.
Let $\psi\in \LMF^2_S(f)$ where $\Psi^0=\begin{pmatrix} u' & v' \\ * & * \end{pmatrix}$.  By Lemma \ref{lem.pp3}, $\phi\cong \psi$ if and only if $\Coker \phi
\cong \Coker \psi
$, however, this does not imply $A/uA+vA\cong A/u'A+v'A$,
so the above statement of Lemma \ref{lem.cocoW} is convenient for the purpose of this paper.
\end{remark}

Let $S=k\<x_1, \dots, x_n\>/I$ be a graded algebra.  For $\Phi=(\phi_{ij})\in M_r(S)$ such that $\phi_{ij}\in S_1$, we define
$$\dim \Phi:=\dim _k\sum_{1\leq i, j\leq r}k\phi_{ij}.$$

\begin{lemma} \label{lem.lisc1}
Let $S=k\<x_1, \dots, x_n\>/I$ be a graded algebra, and $\Phi=(\phi_{ij}), \Psi=(\psi_{ij})\in M_r(S)$ such that $\phi_{ij}, \psi_{ij}\in S_1$.
\begin{enumerate}
\item{} $\dim \Phi^t=\dim \Phi$.
\item{} If $\Phi\sim \Psi$,
then $\dim \Psi=\dim \Phi$.
\end{enumerate}
\end{lemma}

\begin{proof} (1) Clear.

(2) We will prove the case $r=2$, which is needed in this paper.
If $\Phi = \begin{pmatrix} \phi_1 & \phi_2 \\ \phi_3 &\phi_4 \end{pmatrix}$ and $P = \begin{pmatrix} a & b \\ c & d \end{pmatrix}\in \GL_2(k)$, then
$$P\Phi=\begin{pmatrix} a\phi_1+b\phi_3 & a\phi_2+b\phi_4 \\ c\phi_1+d\phi_3 & c\phi_2+d\phi_4 \end{pmatrix},$$
so
$$\begin{pmatrix} a\phi_1+b\phi_3 \\ c\phi_1+d\phi_3\\ a\phi_2+b\phi_4  \\ c\phi_2+d\phi_4 \end{pmatrix}=\begin{pmatrix} P & 0 \\ 0 & P \end{pmatrix}\begin{pmatrix} \phi_1 \\ \phi_3 \\ \phi_2 \\ \phi_4 \end{pmatrix}.$$
Since $\begin{pmatrix} P & 0 \\ 0 & P \end{pmatrix}\in \GL_4(k)$, we have $\dim P\Phi=\dim \Phi$.   For $Q\in \GL_2(k)$,
$$\dim P\Phi Q=\dim \Phi Q=\dim (\Phi Q)^t=\dim Q^t\Phi^t=\dim \Phi^t=\dim \Phi$$
by (1).
\end{proof}

\begin{lemma} \label{lem.4con}
Let $A=S/(f)$ be an irreducible noncommutative quadric.  If $\phi\in \LMF^2_S(f)$, then $\dim \Phi^0 \geq 3$.
\end{lemma}

\begin{proof}
Suppose that $\Phi^0= \begin{pmatrix} \phi_1 & \phi_2 \\ \phi_3 &\phi_4 \end{pmatrix}$
satisfies $\dim \Phi^0\leq 2$.
By Lemma \ref{lem.irr}, $\{\phi_1, \phi_2\}$ is linearly independent.  Since $\dim \Phi^0\leq 2$, $\phi_3=\a \phi_1+\b\phi_2, \phi_4=\c \phi_1+\d \phi_2$ for some $\a, \b, \c, \d\in k$ and $\b\neq 0$.  Then we have
$$\begin{pmatrix} 1 & 0 \\ a & 1 \end{pmatrix}\begin{pmatrix} \phi_1 & \phi_2 \\ \phi_3 &\phi_4 \end{pmatrix}\begin{pmatrix} 1 & b \\ 0 & 1 \end{pmatrix}=\begin{pmatrix} \phi_1 & b\phi_1+\phi_2 \\ a\phi_1+\phi_3 & ab\phi_1+a\phi_2+b\phi_3+\phi_4 \end{pmatrix},$$
and
$$ab\phi_1+a\phi_2+b\phi_3+\phi_4=(b(a+\a)+\c)\phi_1+(a+b\b +\d)\phi_2.$$
It remains to show that there exist $a, b\in k$ such that $b(a+\a)+\c=a+b\b +\d=0$.  To solve this system of equations, since $a=-b\b-\d$, it is enough to find $b\in k$ such that $\b b^2+(\d-\a)b-\c=0$.  Since $\b\neq 0$ and $k$ is algebraically closed, it has a solution.
Therefore $\Phi^0\sim \begin{pmatrix} * & * \\ * & 0 \end{pmatrix}$.
By Lemma \ref{lem.irr}, $f$ is reducible, which is a contradiction.
\end{proof}

\begin{lemma} \label{lem.isom}
Let $A=S/(f)$ be an irreducible noncommutative quadric, and $\phi\in \LMF^2_S(f)$.
For $(a, b), (c, d)\in \PP^1$,  $\Coker (a, b)\overline {\Phi^0}\cong \Coker (c, d)\overline {\Phi^0}$ if and only if $(a, b)=(c, d)$.
\end{lemma}

\begin{proof} Without loss of generality, we may assume that $(a, b)=(1, 0)$.  Let $\Phi^0=\begin{pmatrix} \phi_1 & \phi_2 \\ \phi_3 & \phi_4 \end{pmatrix}$.  If
$$\Coker (1, 0)\overline {\Phi^0}=A/\phi_1A+\phi_2A\cong \Coker (c, d)\overline {\Phi^0}=A/(c\phi_1+d\phi_3)A+(c\phi_2+d\phi_4)A,$$
then $k\phi_1+k\phi_2=k(c\phi_1+d\phi_3)+k(c\phi_2+d\phi_4)$.    By Lemma \ref{lem.4con}, $\dim \Phi^0\geq 3$. If $d\neq 0$, then $\phi_3, \phi_4\in k\phi_1+k\phi_2$, which is a contradiction, so $(c, d)=(1, 0)$.
\end{proof}

\subsection{Rulings on an Irreducible Noncommutative Quadric}

\begin{definition} Let $A=S/(f)$ be an irreducible noncommutative quadric.
For $X=\Coker (\phi[1])(1)\in \mathbb M$ where $\phi\in \LMF_S^2(f)$,
we define  the {\it ruling} associated to $X$ by $${\mathsf R}_X:=\{\Coker ((a, b)\overline {\Phi^0}) \mid (a, b)\in \PP^1\}.$$
\end{definition}

\begin{lemma} \label{lem.coco}
Let $A=S/(f)$ be an irreducible noncommutative quadric.
Then the set of all isomorphism classes of line modules over $A$ is given by
$$\coprod _{X\in \mathbb M}{\mathsf R}_X.$$
Moreover,
$\Omega L\cong X(-1)$ for every $L\in {\mathsf R}_X$.
\end{lemma}

\begin{proof}
By  Lemmas \ref{lem.mphi}, \ref{lem.cocoW}, and \ref{lem.isom}, the set of all isomorphism classes of line modules over $A$ is given by
$$\bigcup _{X\in \mathbb M}{\mathsf R}_X.$$

Since
$$\begin{CD} \cdots @>\overline{\Phi^{1}}\cdot>>A(-1)^2 @>\overline{(a, b)\Phi^{0}}\cdot>>
A \to L\to 0\end{CD}$$
is a (linear) free resolution of $L$,
$\Omega L\cong \Coker \overline{\Phi^{1}}\cong X(-1)$.
If $L\in {\mathsf R}_X\cap {\mathsf R}_Y$ for $X, Y\in \mathbb M$, then $X\cong \Omega L(1)\cong Y$, so the union is disjoint.
\end{proof}

\begin{proposition} \label{prop.321}
Let $A$ be a smooth irreducible noncommutative quadric, and $L, L'$ line modules over $A$.
Then there exists an exact sequence
\begin{equation}\label{eline}
0\to L'(-1)\to A/aA\to L\to 0
\end{equation}
for some $0\neq a\in A_1$ if and only if
$X'\cong \Omega X(1)$, where
$X= \Omega L(1), X'=\Omega L'(1)\in \mathbb M$.
\end{proposition}

\begin{proof}
If $X'\cong \Omega X(1)$,
then, as in the proof of \cite [Proposition 7.4]{SV}, we obtain an exact sequence \eqref{eline}.

Suppose that $X'\ncong \Omega X(1)$ and that there exists an exact sequence \eqref{eline}.
Since $A/aA$ is generated in degree 0 while $L'(-1)$ is generated in degree 1, there is no nonzero map $A/aA\to L'(-1)$, so \eqref{eline} does not split.
To derive a contradiction, it is enough to show that $\Ext^1_A(L, L'(-1))_0=0$.  An exact sequence $0\to X(-1)\to A\to L\to 0$ induces an exact sequence
$$\Hom_A(X(-1), L'(-1))_0\to \Ext^1_A(L, L'(-1))_0\to \Ext_A^1(A, L'(-1))_0=0,$$
so it is enough to show that $\Hom_A(X(-1), L'(-1))_0\cong \Hom_A(X, L')_0=0$.  Since $\depth X=3$, an exact sequence
$0\to X'(-1)\to A\to L'\to 0$ induces an exact sequence
$$
0 =\Hom_A(X, A)_0\to \Hom_A(X, L')_0 \to \Ext^1_A(X, X'(-1))_0,
$$
so it is enough to show that $\Ext^1_A(X, X'(-1))_0=0$.

By Lemma \ref{lem.pp4}(2) and Proposition \ref{prop.sism}(1), we have an equivalence $\uCM^0(A)\cong \mod (M_2(k) \times M_2(k))$. Under this equivalence, $\Omega X(1)$ and $X'$ correspond respectively
to the two non-isomorphic simple modules over $M_2(k) \times M_2(k)$. Therefore we get
\begin{align*}
\Ext^1_A(X, X'(-1))_0 & \cong \Hom_{\uCM^{\ZZ}(A)}(X, X'(-1)[1]) \cong \Hom_{\uCM^{\ZZ}(A)}(\Omega X(1), X')=0.
\end{align*}
Hence the assertion follows.
\end{proof}

\begin{corollary} \label{cor.mq}
Let $A$ be a smooth irreducible noncommutative quadric.
\begin{enumerate}
\item{} $A$ is standard if and only if there exist a line module $L$ over $A$ and an exact sequence
$$0\to L'(-1)\to A/aA\to L\to 0$$
for some $0\neq a\in A_1$ such that $L'$ is a line module over $A$ in a different ruling of $L$.
\item{} $A$ is non-standard if and only if there exist a line module $L$ over $A$ and  an exact sequence
$$0\to L'(-1)\to A/aA\to L\to 0$$
for some $0\neq a\in A_1$ such that $L'$ is a line module in the same ruling of $L$.
\end{enumerate}
\end{corollary}

\begin{proof}
By Proposition \ref{prop.321}, there exists an exact sequence
$$0\to L'(-1)\to A/aA\to L\to 0$$
where $L$ and $L'$ are line modules belonging to different rulings (resp.\ the same ruling)  if and only if there exist $X, Y \in \mathbb M$ such that $Y \cong \Omega X(1)$ and $X \neq Y$ (resp.\ $X=Y$). Thus, such an exact sequence exists if and only if $A$ is standard (resp. non-standard).
\end{proof}

\section{Standardness under Zhang Twists}

In this section, we study how standardness of smooth irreducible noncommutative quadrics behaves under Zhang twists.
We first give a criterion in terms of the induced action on the families of lines associated to $X,Y\in\mathbb M$. We then apply this criterion to twists of the smooth commutative quadric.

\subsection{Twists of a Smooth Irreducible Noncommutative  Quadric}

In this subsection, we study the Zhang twist $A^\sigma$ of a noncommutative
quadric $A=S/(f)$ by an automorphism $\sigma\in \GrAut S$ such that
$\sigma(f)=\lambda f$ for some $0\neq \lambda\in k$.
Then $\sigma$ induces a graded algebra automorphism of $A$, which we still
denote by $\sigma\in \GrAut A$ by abuse of notation.
By adjusting the scalar, we may and will assume that $\sigma(f)=f$.

\begin{lemma}[{\cite[Theorem 5.11]{Zh}}]\label{lem.zh}
For every $\s\in \GrAut S$, $S$ is a quantum polynomial algebra satisfying the condition (*)
if and only if so is $S^\s$.
\end{lemma}

Recall that, for a graded algebra $S$ and $\s\in \GrAut S$, there is an equivalence functor $\GrMod S\to \GrMod S^{\s};\ M\mapsto M^\s$; see \eqref{eZh}.

\begin{lemma}[{\cite[Theorem 3.7]{MUnmf}}] \label{lem.twnm}
Let $S$ be a graded algebra and $f\in S_d$.  For $\s\in \GrAut S$ such that $\s(f)=f$,
$$\NMF_S^{\ZZ}(f)\to \NMF_{S^\s}^{\ZZ}(f^\s); \quad  \phi=\{\phi^i\}_{i\in \ZZ}\mapsto \phi^{\s}:=\{(\phi^i)^{\s}\}_{i\in \ZZ}$$
is an equivalence functor, which restricts to an equivalence functor
$$\LMF_S^r(f)\to \LMF_{S^\s}^r(f^\s).$$
\end{lemma}

Let $f\in S$ be a homogeneous element.  If there exists a bijection $\nu:S\to S$ such that $af=f\nu(a)$ for every $a\in S$, then $f$ is normal.  Moreover, if $f$ is regular normal, then there exists a unique $\nu\in \GrAut S$ such that $af=f\nu(a)$ for every $a\in S$.
In this case, we call $\nu\in \GrAut S$ the {\it normalizing automorphism} of $f$.

\begin{proposition} \label{prop.mq1}
Let $A=S/(f)$ be a smooth irreducible noncommutative quadric with $\mathbb M=\{X, Y\}$,
 and $\s\in \GrAut S$ such that $\s(f)=f$.
\begin{enumerate}
\item{} $A^{\s}$ is a smooth irreducible noncommutative quadric with $\mathbb M^{\s}=\{X^{\s}, Y^{\s}\}$.
\item{}  ${\mathsf R}_{X(-1)^{\s}(1)}=\{L^{\s}\mid L\in {\mathsf R}_X\}, {\mathsf R}_{Y(-1)^{\s}(1)}=\{L^{\s}\mid L\in {\mathsf R}_Y\}$ are two rulings of $A^{\s}$
so that either $X(-1)^{\s}(1)\cong X^\s, Y(-1)^{\s}(1)\cong Y^\s$ or $X(-1)^{\s}(1)\cong Y^\s, Y(-1)^{\s}(1)\cong X^\s$.
\item{} For line modules $L, L'$ over $A$, $L^{\s}, {L'}^{\s}$ are line modules over $A^{\s}$, and $L, L'$ are in the same ruling if and only if $L^{\s}, {L'}^{\s}$ are in the same ruling.
\end{enumerate}
\end{proposition}

\begin{proof} (1)  Since $S$ and $S^\s$ are domains, every non-zero element of $S$ or $S^\s$ is regular.  Since $\s(f)=f$, we have $A^\s\cong S^{\s}/(f^\s)$.  Let  $\nu\in \GrAut S$
be the normalizing automorphism of the (regular) normal element of $f\in S_2$.
For $a^\s\in S^\s_i$, there exists a bijection $\rho:S^\s\to S^\s; a^\s\mapsto(\s^{-2}\nu(a))^\s$ such that
\[ a^\s f^\s =(a\s^i(f))^\s =(af)^\s =(f\nu(a))^\s
=f^\s(\s^{-2}\nu(a))^\s =f^\s\rho(a^\s),
\]
so $f^\s\in S^\s_2$ is a (regular) normal element.
Since $S^\s$ is a $4$-dimensional quantum polynomial algebra,
$f^\s\in S^{\s}_2$ is an irreducible normal element, and $\tails A \cong \tails A^\s$ by \cite[Theorem 1.4]{Zh},
$A^\s$ is a smooth irreducible noncommutative quadric.  If $X\cong \Coker \phi, Y\cong \Coker \psi$ for $\phi, \psi\in \LMF_S^2(f)$, then $X^\s\cong \Coker \phi^\s, Y^\s\cong \Coker \psi^\s$ for $\phi^\s, \psi^\s\in \LMF_{S^\s}^2(f^\s)$ by Lemma \ref{lem.twnm},
so $\mathbb M^\s=\{X^\s,Y^\s\}$ by Lemma \ref{lem.mphi} (1).

(2)
The exact sequence $0\to X(-1)\to A\to L\to 0$ in $\grmod A$ induces an exact sequence $0\to X(-1)^{\s}\to A^{\s}\to L^{\s}\to 0$ in $\grmod A^{\s}$,
so $\Omega L^{\s}(1)\cong X(-1)^{\s}(1)$.  The result now follows from Lemma \ref{lem.coco}.

(3) By (2), $L, L'\in {\mathsf R}_X$ if and only if $L^{\s}, {L'}^{\s}\in {\mathsf R}_{X(-1)^{\s}(1)}$, hence the result follows.
\end{proof}

\begin{remark}  In the above proof, since $f$ is regular and
$$f\nu\s(a)=\s(a)f=\s(a)\s(f)=\s(af)=\s(f\nu(a))=\s(f)\s\nu(a)=f\s\nu(a)$$
for every $a\in S$, we have $\s\nu=\nu\s$.
It follows that the map $\rho:S^\s\to S^\s; a^\s\mapsto(\s^{-2}\nu(a))^\s$
is a graded algebra automorphism of $S^\s$, so $\rho$ is in fact the normalizing automorphism of $f^\s\in S^\s$.
\end{remark}

Let $A=T(V)/I$ be a graded algebra generated by a $4$-dimensional vector space $V$ over $k$. For linearly independent vectors $u, v\in V$, we define $L_{\ell}=A/uA+vA \in \grmod A$, where $\ell=\cV(u,v)\subset \PP(V^*)=\PP^3$ is a line.  (Note that $L_{\ell}$ depends on $\ell=\cV(u, v)$ but not on $u, v$.)  We need the following lemma later.

\begin{lemma}  \label{lem.mq2}
Let $A=T(V)/I$ be a graded algebra generated by a $4$-dimensional vector space $V$ over $k$.
For $\s\in \GrAut A \leq \GL(V)$, we have $L_{\ell}(-1)^{\s}\cong (L_{\s^*(\ell)})^{\s}(-1)$ in $\grmod A^{\s}$ where $\s^*(\ell):=\cV(\s^{-1}(u), \s^{-1}(v))$.
\end{lemma}

\begin{proof}
Since $\s^*(\ell)=\cV(\s^{-1}(u), \s^{-1}(v))\subset \PP^3$, we have an exact sequence
\[
\xymatrix@C=1.2cm{
A(-1)^2 \ar[rr]^-{(\sigma^{-1}(u),\,\sigma^{-1}(v))\cdot}
  && A \ar[r]
  & L_{\sigma^*(\ell)} \ar[r]
  & 0
}
\]
in $\grmod A$, which induces an exact sequence
\[
\xymatrix@C=1.2cm@R=.6cm{
A^2(-1)^{\s} \ar[rr]^-{(\s^{-1}(u),\,\s^{-1}(v))\cdot^{\s}}
  && A^{\s} \ar[r]
  & (L_{\s^*(\ell)})^{\s} \ar[r]
  & 0 \\
A^{\s}(-1)^2 \ar[u]^{(\s,\,\s)}_{\cong} \ar[rr]^-{(\s^{-1}(u),\,\s^{-1}(v))\cdot}
  && A^{\s} \ar[u]^{\id}_{\cong}
}
\]
in $\grmod A^{\s}$, so $(L_{\s^*(\ell)})^{\s}\cong A^{\s}/(\s^{-1}(u)A^{\s}+\s^{-1}(v)A^{\s})$.
On the other hand, an exact sequence
\[
\xymatrix@C=1.2cm{
A(-2)^2 \ar[r]^-{(u,\,v)\cdot}
  & A(-1) \ar[r]
  & L_{\ell}(-1) \ar[r]
  & 0
}
\]
in $\grmod A$ induces an exact sequence
\[
\xymatrix@C=1.2cm@R=.6cm{
A^2(-2)^{\s} \ar[rr]^{(u,\,v)\cdot^{\s}}
  && A(-1)^{\s} \ar[r]
  & L_{\ell}(-1)^{\s} \ar[r]
  & 0 \\
A^{\s}(-2)^2 \ar[u]^{(\s^2,\,\s^2)}_{\cong}
  \ar[rr]^{(\s^{-1}(u),\,\s^{-1}(v))\cdot}
  && A^{\s}(-1) \ar[u]^{(\s,\,\s)}_{\cong}
}
\]
in $\grmod A^{\s}$, so
$L_{\ell}(-1)^{\s}\cong (A^{\s}/(\s^{-1}(u)A^{\s}+\s^{-1}(v)A^{\s}))(-1) \cong (L_{\s^*(\ell)})^{\s}(-1)$.
\end{proof}

For the rest of this section, we assume that  $A = S/(f)$ is a standard smooth irreducible noncommutative quadric
generated by $x,y,z,w$.
Let $\mathbb M=\{X, Y\}$ and let $\nu\in \GrAut S$ be the normalizing automorphism of $f$.
We choose
$$\left(\Phi=\begin{pmatrix} \phi_1 & \phi_2 \\ \phi_3 &\phi_4 \end{pmatrix},\Psi=\begin{pmatrix} \psi_1 & \psi_2 \\ \psi_3 &\psi_4 \end{pmatrix}\right)\in \LMF_S^2(f)$$
such that $X=\Coker \overline{\Phi}$ and $Y=\Coker \overline {\Psi}$. We define
\[
\Phi^\natural = \begin{pmatrix} \phi_1 & \phi_2 \\ \phi_3 &\phi_4 \end{pmatrix}
\;\;\text{and}\;\;
\Psi^\natural = \begin{pmatrix} \psi_1 & \psi_2 \\ \psi_3 &\psi_4 \end{pmatrix}
\;\;
\in M_2(k[x,y,z,w]_1)
\]
Note that if $\Phi' \in M_2(S_1)$ such that $\Phi\sim \Phi'$, then $\Phi'=P\Phi Q$ for some $P, Q\in \GL_2(k)$, so $\cV(\det (\Phi')^\natural)=\cV(\det P\Phi^\natural Q)=\cV(\det \Phi^\natural)\subset \PP^3$.

If $L \in {\mathsf R}_X$, then we have
\[ L \cong (\Coker (a,b)\overline{\Psi})(1)\cong A/(a\psi_1+b\psi_3)A+(a\psi_2 +b\psi_4)A =:L_p\]
for some unique $p= (a,b) \in \PP^1$ by Lemma \ref{lem.coco}.
For $L_p \in {\mathsf R}_X$, we define
\begin{align}
	\ell_p := \cV(a\psi_1+b\psi_3, a\psi_2 +b\psi_4) \;\; \subset \cV(\det \Psi^\natural)= \cV(\psi_1\psi_4-\psi_2\psi_3) \label{lp}
\subset \PP^3
\end{align}
and ${\mathsf L}_X :=\{\ell_p \subset \PP^3\mid p \in \PP^1\}$.

Similarly, if $L \in {\mathsf R}_Y$, then we have
\[L \cong (\Coker (a,b)\overline{\Phi}) (1) \cong A/(a\phi_1+b\phi_3)A+(a\phi_2 +b\phi_4)A =: {_p}L\]
for some unique $p= (a,b) \in \PP^1$ by Lemma \ref{lem.coco}.
For ${_p}L \in {\mathsf R}_Y$, we define
\[ {_p}\ell := \cV(a\phi_1+b\phi_3, a\phi_2 +b\phi_4)
\;\; \subset \cV(\det \Phi^\natural)= \cV(\phi_1\phi_4-\phi_2\phi_3) \subset \PP^3\]
and ${\mathsf L}_Y :=\{{_p}\ell \subset \PP^3\mid p \in \PP^1\}$. It is easy to see that ${\mathsf L}_X$ and ${\mathsf L}_Y$ are well-defined.

\begin{remark}
If $A$ is commutative, then $\cV(\det \Phi^\natural)=\cV(\det\Psi^\natural)=\Proj A\subset \PP^3$, however,  if $A$ is not commutative, then $\cV(\det \Phi^\natural)$ and $\cV(\det\Psi^\natural)$ do not always coincide, so ${\mathsf L}_X$ and ${\mathsf L}_Y$ may not be two rulings of the same smooth quadric in $\PP^3$ (see Remark \ref{rem.det}).
\end{remark}

\begin{lemma} \label{lem.empty}
With the notation above, ${\mathsf L}_X \cap {\mathsf L}_Y=\varnothing$.
\end{lemma}

\begin{proof}
Suppose that there exists $\ell \subset \PP^3$ such that $\ell \in {\mathsf L}_X \cap {\mathsf L}_Y$.
Then
\[ \cV(a\phi_1+b\phi_3, a\phi_2 +b\phi_4) = \ell = \cV(c\psi_1+d\psi_3, c\psi_2 +d\psi_4). \]
Since $\phi_i, \psi_j$ are of degree 1, there exists $\begin{pmatrix} p_1 &p_2 \\ p_3 &p_4 \end{pmatrix} \in \GL_2(k)$ such that
\[ c\psi_1+d\psi_3 = p_1(a\phi_1+b\phi_3) + p_3(a\phi_2 +b\phi_4), \quad  c\psi_2 +d\psi_4 = p_2(a\phi_1+b\phi_3) + p_4(a\phi_2 +b\phi_4) \]
in $k[x, y, z, w]_1$ and in $A_1$. We have
\begin{align*}
{\mathsf R}_X \ni &A/(c\psi_1+d\psi_3)A+ (c\psi_2 +d\psi_4)A\\
&= A/(p_1(a\phi_1+b\phi_3) + p_3(a\phi_2 +b\phi_4))A+(p_2(a\phi_1+b\phi_3) + p_4(a\phi_2 +b\phi_4))A\\
&= A/(a\phi_1+b\phi_3)A+ (a\phi_2 +b\phi_4)A \in {\mathsf R}_Y,
\end{align*}
so $X \cong Y$, which is a contradiction. It follows that ${\mathsf L}_X \cap {\mathsf L}_Y = \varnothing$.
\end{proof}

Let $\s \in \GrAut S$ be such that $\s(f) =f$. We define
\[
\s^{-1}(\Phi)=
\begin{pmatrix} \s^{-1}(\phi_1) & \s^{-1}(\phi_2) \\ \s^{-1}(\phi_3) &\s^{-1}(\phi_4) \end{pmatrix}
\ \ \text{and} \ \
\s^{-1}(\Psi)=
\begin{pmatrix} \s^{-1}(\psi_1) & \s^{-1}(\psi_2) \\ \s^{-1}(\psi_3) &\s^{-1}(\psi_4) \end{pmatrix}.
\]
Then
\begin{align*}
\s^{-1}(\Phi)\s^{-1}(\Psi)
&=\begin{pmatrix} \s^{-1}(\phi_1) & \s^{-1}(\phi_2) \\ \s^{-1}(\phi_3) &\s^{-1}(\phi_4) \end{pmatrix} \begin{pmatrix} \s^{-1}(\psi_1) & \s^{-1}(\psi_2) \\ \s^{-1}(\psi_3) &\s^{-1}(\psi_4) \end{pmatrix}\\
&=\begin{pmatrix} \s^{-1}(f) & \s^{-1}(0) \\ \s^{-1}(0) &\s^{-1}(f) \end{pmatrix} = \begin{pmatrix} f & 0 \\0 &f \end{pmatrix}
\end{align*}
in $M_2(S)$, so $(\s^{-1}(\Phi), \s^{-1}(\Psi))\in \LMF_S^2(f)$.
It follows that
$\Coker \overline {\s^{-1}(\Phi)}\in \mathbb{M}$, so either
\begin{itemize}
\item[($\dagger$)] $\Coker \overline {\s^{-1}(\Phi)} \cong X=\Coker \overline {\Phi}$, or
\item[($\ddagger$)] $\Coker \overline {\s^{-1}(\Phi)} \cong Y=\Coker \overline{\Psi}$
\end{itemize}
occurs.  Consider
\[ \cV(\det {\s^{-1}(\Phi)}^\natural) = \cV(\s^{-1}(\phi_1)\s^{-1}(\phi_4)-\s^{-1}(\phi_2)\s^{-1}(\phi_3))\; \subset \PP^3. \]
We then define
\begin{align}
\begin{split}
\s^* : \cV(\phi_1\phi_4 -\phi_2\phi_3) &\to \cV(\s^{-1}(\phi_1)\s^{-1}(\phi_4)-\s^{-1}(\phi_2)\s^{-1}(\phi_3)); \label{snatural}\\
p \;\; &\mapsto \;\; (\s(x)(p), \s(y)(p), \s(z)(p), \s(w)(p)).
\end{split}
\end{align}
Note that $\s^*$ is an isomorphism because $\s$ is induced by a matrix of $\GL_4(k)$.

\begin{lemma} \label{lem.dsd}
Let $A=S/(f)$ be a standard smooth irreducible noncommutative quadric and let $\s\in \GrAut S$ such that $\s(f)=f$.
If ($\dagger$) occurs, then  $\s^*(\ell_p)\in {\mathsf L}_Y$ for every $\ell_p \in {\mathsf L}_Y$.
If ($\ddagger$) occurs, then $\s^*(\ell_p)\in {\mathsf L}_X$ for every $\ell_p\in {\mathsf L}_Y$.
\end{lemma}

\begin{proof}
Let $\ell_p = \cV(a\phi_1+b\phi_3, a\phi_2 +b\phi_4) \in {\mathsf L}_Y$. Then we have
\begin{align*}
\s^*(\ell_p) &= \cV(\s^{-1}(a\phi_1+b\phi_3), \s^{-1}(a\phi_2 +b\phi_4))\\
&= \cV(a \s^{-1}(\phi_1)+b\s^{-1}(\phi_3), a\s^{-1}(\phi_2) +b\s^{-1}(\phi_4))\\
&= \cV((a,b) \s^{-1}(\Phi)^\natural) \quad \subset \cV(\det \s^{-1}(\Phi)^\natural).
\end{align*}

If ($\dagger$) occurs, then $\s^{-1}(\Phi) = P \Phi Q$ for some $P, Q \in \GL_2(k)$ by Lemma \ref{lem.pp3}, so
\begin{align*}
\cV((a,b) \s^{-1}(\Phi)^\natural)
&= \cV((a,b) P \Phi^\natural Q)\\
&=\cV((c\phi_1+d\phi_3, c\phi_2 +d\phi_4)Q)\\
&=\cV(c\phi_1+d\phi_3, c\phi_2 +d\phi_4) \quad \subset \cV(\det \Phi^\natural)
\end{align*}
where $(c,d) =(a,b)P \in \PP^1$, so $\s^*(\ell_p) \in {\mathsf L}_Y$.

If ($\ddagger$) occurs, then $\s^{-1}(\Phi) = P \Psi Q$ for some $P, Q \in \GL_2(k)$ by Lemma \ref{lem.pp3}, so
\begin{align*}
\cV((a,b) \s^{-1}(\Phi)^\natural) &= \cV((a,b) P \Psi^\natural Q)\\
&=\cV((c\psi_1+d\psi_3, c\psi_2 +d\psi_4) Q)\\
&=\cV(c\psi_1+d\psi_3, c\psi_2 +d\psi_4) \quad \subset \cV(\det \Psi^\natural)
\end{align*}
where $(c,d) =(a,b)P \in \PP^1$, so $\s^*(\ell_p) \in {\mathsf L}_X$.
\end{proof}

\begin{lemma} \label{lem.dsn}
Let  $A=S/(f)$ be a standard smooth irreducible noncommutative quadric and let $\s\in \GrAut S$ such that $\s(f)=f$.
If ($\dagger$) occurs, then $A^{\s}$ is standard.
If ($\ddagger$) occurs, then $A^{\s}$ is non-standard.
\end{lemma}

\begin{proof}
Since $A$ is standard, Corollary \ref{cor.mq} implies that there exist
line modules $L$ and $L'$ in different rulings and $0\neq a\in A_1$ such that
\[
0\to L'(-1) \to A/aA \to L \to 0
\]
is an exact sequence in $\grmod A$.
Put $L' = A/uA+ vA$. Then we have an exact sequence
\[
\xymatrix@C=.6cm@R=.6cm{
0 \ar[r]
  & (A/uA+vA)(-1)^{\s} \ar[r]\ar[d]^{\cong}
  & (A/aA)^{\s} \ar[r] \ar[d]^{\cong}
  & L^{\s} \ar[r]
  & 0 \\
  & (A/(\s^{-1}(u)A+\s^{-1}(v)A))^{\s}(-1)
  & A^{\s}/a^{\s}A^{\s}
}
\]
in $\grmod A^\s$ by Lemma \ref{lem.mq2}.

If ($\dagger$) occurs, then $\Coker \overline {\s^{-1}(\Phi)} \cong \Coker \overline {\Phi}$ and $\Coker \overline {\s^{-1}(\Psi)} \cong \Coker \overline {\Psi}$.
It follows that $L'=A/uA+ vA$ and $A/(\s^{-1}(u)A+ \s^{-1}(v)A)$ are in the same ruling,
and so $L$ and $A/(\s^{-1}(u)A+\s^{-1}(v)A)$ are in different rulings.
By Proposition \ref{prop.mq1}(3), $L^{\s}$ and $(A/(\s^{-1}(u)A+\s^{-1}(v)A))^{\s}$ are in  different rulings, so  $A^\s$ is standard by Corollary \ref{cor.mq}.

If ($\ddagger$) occurs, then
$\Coker \overline {\s^{-1}(\Phi)} \cong \Coker \overline {\Psi}$ and $\Coker \overline {\s^{-1}(\Psi)} \cong \Coker \overline{\Phi}$.
It follows that $L'=A/uA+ vA$ and $A/(\s^{-1}(u)A+ \s^{-1}(v)A)$ are in different rulings,
and so $L$ and $A/(\s^{-1}(u)A+ \s^{-1}(v)A)$ are in the same ruling.
By Proposition \ref{prop.mq1}(3), $L^{\s}$ and $(A/(\s^{-1}(u)A+\s^{-1}(v)A))^{\s}$ are in the same ruling, so $A^\s$ is non-standard by Corollary \ref{cor.mq}.
\end{proof}

As a consequence, we obtain the following result.

\begin{theorem}\label{thm:switch}
Let  $A=S/(f)$ be a standard smooth irreducible noncommutative quadric.
For any $\s\in \GrAut S$ such that $\s(f)=f$,
$A^{\s}$ is standard if and only if $\s^*(\ell_p)\in {\mathsf L}_Y$ for every $\ell_p \in {\mathsf L}_Y$, and $A^{\s}$ is non-standard if and only if $\s^*(\ell_p)\in {\mathsf L}_X$ for every $\ell_p \in {\mathsf L}_Y$.
\end{theorem}

\begin{proof}
Assume that $A^{\s}$ is standard. Then ($\ddagger$) does not occur by Lemma \ref{lem.dsn}, so ($\dagger$) occurs.
Thus $\s^*(\ell_p) \in {\mathsf L}_Y$ for every $\ell_p \in {\mathsf L}_Y$ by Lemma \ref{lem.dsd}.

Conversely, assume that $\s^*(\ell_p) \in {\mathsf L}_Y$ for every $\ell_p \in {\mathsf L}_Y$.
Since ${\mathsf L}_X \cap {\mathsf L}_Y = \varnothing$ by Lemma \ref{lem.empty}, we see $\s^*(\ell_p) \not \in {\mathsf L}_X$, so ($\ddagger$) does not occur by Lemma \ref{lem.dsd}.
Thus ($\dagger$) occurs. It follows from Lemma \ref{lem.dsn} that $A^{\s}$ is standard.

The assertion for the non-standard case is proved in the same way.
\end{proof}

\subsection{Twists of the Smooth Commutative Quadric}

In this subsection, we study a twisted algebra  $(k[x, y, z, w]/(f))^\s$ of a smooth commutative quadric.
Note that $\s\in \GrAut k[x, y, z, w]$ such that $\s(f)=\l f$ for some $0\neq \l\in k$ is essentially the same as $\s\in \GrAut k[x, y, z, w]/(f)$.
Recall that every smooth commutative quadric is standard irreducible.
By Lemma \ref{lem.zh}, $k[x, y, z, w]^\s$ is a quantum polynomial algebra satisfying the condition (*).

For a projective variety $X\subset \PP^{n-1}$, we define
$$\Aut (X\uparrow \PP^{n-1}):=\{\s|_X\in \Aut X\mid \s\in \Aut \PP^{n-1}\}.$$
Note that
$\Aut (X\uparrow \PP^{n-1})\cong \GrAut k[X]/k^{\times}$ holds
for any projective variety $X\subset \PP^{n-1}$.

Let $s:\PP^1\times \PP^1\to \PP^3$ be the Segre embedding, that is, $s((a, b), (c, d))=(ac, ad, bc, bd)$.  Then $Q=\Im s=\cV(xw-yz)$
 is the unique smooth quadric surface in $\PP^3$ up to isomorphism.
Let $A=S/(f)$ where $S=k[x, y, z, w]$ and $f=xw-yz\in S_2$ so that $A$ is the homogeneous coordinate ring of the smooth quadric surface $Q$ in $\PP^3$. If
$$\Phi=\begin{pmatrix} x & y \\ z & w \end{pmatrix} \quad \textrm{and} \quad \Psi=\begin{pmatrix} w & -y \\ -z & x \end{pmatrix},$$
then
$(\Phi, \Psi), (\Psi, \Phi)\in \LMF_S^2(f)$ are non-isomorphic noncommutative linear matrix factorizations of $f$, so
${\mathbb M}=\{X:=\Coker \overline{\Phi}, Y:=\Coker \overline{\Psi}\}$.

For a point $p=(a, b)\in \PP^1$,
\begin{align*}
& s(\PP^1\times p)=\cV(bx-ay, bz-aw)
=\ell_p\subset\PP^3, \\
& s(p\times \PP^1)=\cV(bx-az, by-aw)={_{(b,-a)}\ell}\subset \PP^3
\end{align*}
are lines in $Q$, so  ${\mathsf L}_X = \{\ell_p\mid p\in \PP^1\}, {\mathsf L}_Y = \{ {_p\ell}\mid p\in \PP^1\}$ are two rulings on $Q$ (see (\ref{lp})).

Let $Q=\cV(xw-yz)\subset \PP^3$ and $\s\in \Aut (Q\uparrow \PP^3)$.  Since $\s$ sends a line to a line in $\PP^3$, for every $p\in \PP^1$,  there exists $\a(p)\in \PP^1$ such that $\s({_p\ell})={_{\a(p)}\ell}$ or $\s({_p\ell})=\ell_{\a(p)}$.  If $p\neq p'\in \PP^1$, then ${_p\ell}\cap {_{p'}\ell}=\varnothing$, so $\s({_p\ell})\cap \s({_{p'}\ell})=\varnothing$.  It follows that if $\s({_p\ell})={_{\a(p)}\ell}$ for some $p\in \PP^1$, then $\s({_{p}\ell})={_{\a(p)}\ell}$ for every $p\in \PP^1$, and if $\s({_p\ell})=\ell_{\a(p)}$ for some $p\in \PP^1$, then $\s({_{p}\ell})=\ell_{\a(p)}$ for every $p\in \PP^1$ (cf. Lemma \ref{lem.dsd}). We say that $\s$ {\it preserves the rulings} if the former occurs and {\it switches the rulings} if the latter occurs.

\begin{remark} It is well known that there are natural isomorphisms
$$
\Aut (Q\uparrow \PP^3)\cong \Aut Q\cong \Aut (\PP^1\times \PP^1)\cong (\Aut \PP^1\times \Aut \PP^1)\rtimes \<\t\>,$$
where $\t:\PP^1\times \PP^1\to \PP^1\times \PP^1; \; (p, q)\mapsto (q, p)$ is an automorphism of order 2.
For $\s\in \GrAut A$, $\s^*\in \Aut (Q\uparrow \PP^3)$ preserves the rulings if and only if $s^{-1}\s^*s\in \Aut \PP^1\times \Aut \PP^1$.
\end{remark}

The following theorem shows that, for $A=k[x,y,z,w]/(xw-yz)$, exactly half of the noncommutative quadrics of the form $A^\s$ are standard, while the
other half are non-standard.

\begin{theorem} \label{thm.cotw}
	Let $A=k[x, y, z, w]/(xw-yz)$ be a smooth quadric.
	For $\s\in \GrAut A$,
	$A^{\s}$ is standard if and only if $\s^*\in \Aut(Q\uparrow \PP^3)$ preserves the rulings,
	and $A^{\s}$ is non-standard if and only if $\s^*\in \Aut(Q\uparrow \PP^3)$ switches the rulings.
\end{theorem}

\begin{proof}
Since $\Aut (Q\uparrow \PP^3)\cong \Aut Q$ and $\cV(\det \Phi^\natural)=\cV(\det \s^{-1}(\Phi)^\natural)=Q$, the automorphism $\sigma^*$ coincides with the one defined in
\eqref{snatural} in the previous subsection.
Then the result follows from Theorem \ref{thm:switch}.
\end{proof}

Let $S=T(V)/(R)$ be a quadratic algebra, where $R\subset V\otimes V$, and let $f \in S_2$.
If $\s\in \GrAut S$ such that $\s(f)=f$, then the following facts are known:
\begin{itemize}
\item $S^\s \cong T(V)/((\id \otimes \s^{-1})(R))$.
\item Under the above isomorphism, $f^\s \in S^\s$ corresponds to $(\id \otimes \s^{-1})(f) \in T(V)/((\id \otimes \s^{-1})(R))$.
\end{itemize}
For now, we identify
\[
\begin{array}{ccc}
 S^\s  & = & T(V)/((\id \otimes \s^{-1})(R)) \\
\rotatebox[origin=c]{90}{$\in$} & & \rotatebox[origin=c]{90}{$\in$} \\
f^\s & = & (\id \otimes \s^{-1})(f).
\end{array}
\]
With this interpretation,  $(\Phi, \Psi)\in \LMF^2_S(f)$ implies $(\Phi, \s^{-1}(\Psi)) \in \LMF^2_{S^{\s}}(f^\s)$. In fact,
\begin{align*}
\Phi\s^{-1}(\Psi)
&=\begin{pmatrix} \phi_1 & \phi_2 \\ \phi_3 & \phi_4 \end{pmatrix} \begin{pmatrix} \s^{-1}(\psi_1) & \s^{-1}(\psi_2) \\ \s^{-1}(\psi_3) & \s^{-1}(\psi_4) \end{pmatrix}
=\begin{pmatrix} (\id\otimes \s^{-1})(f) & 0 \\ 0 & (\id\otimes \s^{-1})(f) \end{pmatrix}
\end{align*}
in $M_2(S^{\s})$. Using this, we give an example.

\begin{example} \label{ex.coex}
Let $S=k[x, y, z, w], f=xw-yz$, and $A=S/(f)$ so that
$$\left (\Phi:=\begin{pmatrix} x & y \\ z & w \end{pmatrix}, \Psi:=\begin{pmatrix} w & -y \\ -z & x \end{pmatrix}\right)\in \LMF_{S}^2(f).$$
If $\s\in \GrAut S$ is defined by $\s(x)=w, \s(y)=y, \s(z)=z, \s(w)=x$, then $\s(f)=f$
so that $\s\in \GrAut A$.   Under the above identification, we have
$$\left (\Phi=\begin{pmatrix} x & y \\ z & w \end{pmatrix}, \s^{-1}(\Psi)=\begin{pmatrix} x & -y \\ -z & w \end{pmatrix}\right)\in \LMF_{S^\s}^2(f^\s).$$
We can check that
$$S^{\s}=k\<x, y, z, w\>/(xy-yw, xz-zw, x^2-w^2, yz-zy, yx-wy, zx-wz),$$
and $f^{\s}=x^2-yz\in Z(S^{\s})_2$ is a regular central element.
Since $$\begin{pmatrix} 1 & 0 \\ 0 & -1 \end{pmatrix}\s^{-1}(\Psi)\begin{pmatrix} 1 & 0 \\ 0 & -1 \end{pmatrix}=\begin{pmatrix} 1 & 0 \\ 0 & -1 \end{pmatrix}\begin{pmatrix} x & -y \\ -z & w \end{pmatrix}\begin{pmatrix} 1 & 0 \\ 0 & -1 \end{pmatrix}=\Phi,$$
$A^{\s}$ is a non-standard smooth irreducible noncommutative quadric.

We may also check this by geometry.  Let $Q=\cV(f)\subset \PP^3$.  For $p=(a, b)\in \PP^1$,
$$\s^*(\ell_p)=\s^*(\cV(bx-ay, bz-aw))=\cV(bw-ay, bz-ax)={_q\ell}$$
where $q=(-a, b)\in \PP^1$,
so $\s^*\in \Aut (Q\uparrow \PP^3)$ switches the rulings, hence $A^{\s}$ is a non-standard smooth irreducible noncommutative quadric by Theorem \ref{thm.cotw}.
\end{example}

\section{Sklyanin Quadrics}

In this section, we study noncommutative quadrics $A=S/(f)$, where $S$ is a $4$-dimensional (non-degenerate) Sklyanin algebra. We call such an algebra $A$ a {\it Sklyanin quadric}, and call it a {\it central Sklyanin quadric} if $f\in Z(S)_2$.
We will define a 4-dimensional Sklyanin algebra explicitly later, but, for now, we refer to the notation used in \cite {SV}, in particular, we denote a 4-dimensional Sklyanin algebra by $S=\cA(E, \cL, \t)$ where $E\subset \PP^3$ is an elliptic curve, $\cL\in \Pic E$ and $\t\in E$.

\begin{lemma}[{\cite[Corollary 1.9]{LS}}]
Every 4-dimensional Sklyanin algebra $S=\cA(E, \cL, \t)$ is a quantum polynomial algebra satisfying the condition (*).
\end{lemma}

Let $S=\cA(E,\cL,\t)$ be a 4-dimensional Sklyanin algebra. Since many results in the literature are proved under the assumption that $|\t|=\infty$, we assume, until further notice, that $|\t|=\infty$ in order to ensure that these results apply.

For $p, q\in E$, we define $L(p, q):=S/WS$ where $W\subset S_1$ is the subspace of linear forms vanishing on the line $\overline {pq}$.  For $z\in E$, we define
$${\mathsf R}_z:=\{L(p, q)\mid p, q\in E, p+q=z\}.$$
For each $z\in E$, there exists a central element $\Omega (z)\in Z(S)_2$ (unique up to non-zero scalar) with the property that $\Omega (z)\cdot L(p, q)=0$ if and only if $p+q=z$ or $p+q=-z-2\t$.

\begin{theorem}[{\cite[Theorem 10.2]{SV}}] \label{thm.10.2}
 Let $S=\cA(E, \cL, \t)$ be a 4-dimensional Sklyanin algebra and $A=S/(\Omega (z))$ where $z\in E$.  Then $A$ is smooth if and only if $z+\t\not \in E_2$, so that four singular central Sklyanin quadrics are given by $S/(\Omega (\omega-\t))$ where $\omega\in E_2$.
\end{theorem}

\begin{lemma} Let $S=\cA(E, \cL, \t)$ be a 4-dimensional Sklyanin algebra and $A=S/(\Omega (z))$ where $z\in E$.  If  $A$ is smooth (so that $z+\t\not \in E_2$),
then the set of isomorphism classes of line modules over $A$ is given by the disjoint union of isomorphism classes of two families ${\mathsf R}_z$ and ${\mathsf R}_{-z-2\t}$ parametrized by $E/\pm \cong \PP^1$.
\end{lemma}

Note that
${\mathsf R}_z={\mathsf R}_{-z-2\t}$ if and only if $z+\t\in E_2$.

\begin{proof} This follows from \cite [Section 6]{LS}.
\end{proof}

\begin{lemma}  \label{lem.SS}
Let $S=\cA(E, \cL, \t)$ be a 4-dimensional Sklyanin algebra and $A=S/(\Omega (z))$ where $z\in E$. For $L(p, q)\in {\mathsf R}_z$ and $r, s\in E$ such that $p+q+r+s=0$, there exists an exact sequence
$$0\to L(r-\t, s-\t)(-1)\to A/aA\to L(p, q)\to 0$$
where $L(r-\t, s-\t)\in {\mathsf R}_{-z-2\t}$ and $\cV(a)\subset \PP^3$ is a secant plane spanned by $p, q, r, s\in E$.
\end{lemma}

\begin{proof} This follows from the proof of \cite [Lemma 4.5]{SS}.
\end{proof}

The next theorem indicates that ``generic'' smooth noncommutative quadrics are standard.

\begin{theorem}\label{thm.Skl} Let $S=\cA(E, \cL, \t)$ be a 4-dimensional Sklyanin algebra and $A=S/(\Omega (z))$ where $z\in E$.  If  $A$ is smooth, then
$A$ is standard.
\end{theorem}

\begin{proof}  (See the proof of \cite [Proposition 10.1]{SV}.)  Suppose that $A=S/(\Omega (z))$ is non-standard smooth.  By Theorem \ref{thm.10.2}, $z+\t\not \in E_2$.

If $L:=L(p, q), L':=L(p', q')$ are in the same family, then there exist $r, s\in E$ such that $p+q+r+s=z+r+s=p'+q'+r+s=0$.  If $\cV(a), \cV(b)\subset \PP^3$ are secant planes spanned by $p, q, r, s$ and $p', q', r, s$, then there exists a line module $L'':=L(r-\t, s-\t)$ such that
\begin{align*}
0\to L''(-1)\to A/aA\to L\to 0, \\
0\to L''(-1)\to A/bA\to L'\to 0
\end{align*}
are exact sequences by Lemma \ref{lem.SS},
so $L, L', L''$ are in the same ruling by Corollary \ref{cor.mq}.

If $L=L(p, q), L'=L(p', q')$ are in different families, then $p+q+p'+\t+q'+\t=0$.  If $\cV(a)\subset \PP^3$ is a secant plane spanned by $p, q, p'+\t, q'+\t$, then there exists an exact sequence
$$0\to L'(-1)\to A/aA\to L\to 0$$
by Lemma \ref{lem.SS},
so $L, L'$ are in the same ruling by Corollary \ref{cor.mq}.

This contradicts the fact that there are two rulings.
\end{proof}

For the rest, we will show that every smooth central Sklyanin quadric is in fact standard (without the assumption $|\t|=\infty$) by explicit computations, identifying four singular central Sklyanin quadrics.

\begin{lemma} Let $S$ be a quantum polynomial algebra and $f \in Z(S)_2$. If there is an indecomposable matrix factorization $(\Phi^0, \Phi^1)\in \NMF_S^{\ZZ}(f)$ of rank $r$ such that $(\Phi^0, \Phi^1)\cong (\Phi^1, \Phi^0)$, then there exists an indecomposable matrix factorization $(\Psi, \Psi)\in \NMF_S^{\ZZ}(f)$ of rank $r$ (such that $(\Phi^0, \Phi^1)\cong (\Psi, \Psi)$).
\end{lemma}

\begin{proof} This follows from the proof of \cite[Proposition 3.12]{CKMW}.
\end{proof}

\begin{lemma} \label{lem.isns}
Let $A=S/(f)$ be an irreducible noncommutative quadric.  If  $f \in Z(S)_2$, then we have the following criteria:
\begin{enumerate}
\item{} $A$ is singular if and only if there exists a unique $(\Phi, \Phi)\in \LMF_S^2(f)$ up to isomorphism.
\item{} $A$ is standard smooth if and only if there exists $(\Phi, \Psi)\in \LMF_S^2(f)$ such that $(\Phi, \Psi)\not \cong (\Psi, \Phi)$.
\item{} $A$ is non-standard smooth if and only if there exist $(\Phi, \Phi), (\Psi, \Psi)\in \LMF_S^2(f)$ such that $(\Phi, \Phi)\not \cong (\Psi, \Psi)$.
\end{enumerate}
\end{lemma}

\begin{proof}
This lemma follows from Proposition \ref{prop.sism}
and Lemma \ref{lem.mphi}(1).
\end{proof}

	Let $(\alpha_1,\alpha_2,\alpha_3)\in k^3$ be parameters satisfying
	$
	\alpha_1 + \alpha_2 + \alpha_3 + \alpha_1\alpha_2\alpha_3 = 0
	$ and $\alpha_i\neq 0,1,-1$ for all $i=1,2,3$. The $4$-dimensional (non-degenerate) Sklyanin algebra
	$S = S(\alpha_1,\alpha_2,\alpha_3)$ is presented as the quadratic algebra
	\[
	S(\alpha_1,\alpha_2,\alpha_3)
	=
	k\langle x_0,x_1,x_2,x_3\rangle
	\big/
	(r^-_i,\,r^+_i \mid i=1,2,3),
	\]
	where
	\begin{align*}
		r^-_i &:= x_0x_i - x_i x_0 - \alpha_i(x_j x_k + x_k x_j),\\
		r^+_i &:= x_0x_i + x_i x_0 - (x_j x_k - x_k x_j)
	\end{align*}
	for $(i, j, k)\in \mathrm{Cyc}(1,2,3):=\{(1, 2, 3), (2, 3, 1), (3, 1, 2)\}$.

\begin{lemma}[{\cite[Corollary 3.9]{SS2}}]
	If $S(\alpha_1,\alpha_2,\alpha_3)$ is a $4$-dimensional Sklyanin algebra, then
	\[
	\Omega_1 := x_0^2 - x_1^2 - x_2^2 - x_3^2
\;\;\text{and}\;\;\,
	\Omega_2 := (1+\alpha_3)x_1^2
	+ (1+\alpha_1\alpha_3)x_2^2
	+ (1-\alpha_1)x_3^2
	\]
	are central elements of $S(\alpha_1,\alpha_2,\alpha_3)$.
\end{lemma}

\begin{remark} \label{lem.skce}
Let $S=S(\alpha_1,\alpha_2,\alpha_3)$ be a $4$-dimensional Sklyanin algebra satisfying the above non-degenerate parameter conditions, namely $\alpha_i\neq 0,\pm1$ for all $i$.
Then $S$ is of the form $S=\cA(E,\cL,\t)$ with $|\t|\neq 2$; this follows from the description of the associated geometric data in \cite[Section 2]{SS2}.
In this case, we have
\[
Z(S)_2=k\Omega_1+k\Omega_2.
\]
Indeed, if $|\t|=\infty$, then $Z(S)=k[\Omega_1,\Omega_2]$ by \cite[Proposition 6.12]{LS}, so $Z(S)_2=k\Omega_1+k\Omega_2$.
If $2\neq |\t|<\infty$, then $\dim_k Z(S)_2=2$ by \cite[Theorem 4.6]{SmT}. Since $\Omega_1$ and $\Omega_2$ are linearly independent, it follows that $Z(S)_2=k\Omega_1+k\Omega_2$.
\end{remark}

\begin{lemma}[{\cite[Proposition 2.1]{CS}}]\label{lem_auto4skl}
	The maps $\delta_i$ for $i=1, 2, 3$
	defined by
	\[\delta_i(x_0) = \sqrt{\alpha_j\alpha_k} x_i,\;\;\,
	\delta_i(x_i) = -\sqrt{-1} x_0,\;\;\,
	\delta_i(x_j) = - \sqrt{-\alpha_j} x_k,\;\;\,
	\delta_i(x_k) = - \sqrt{\alpha_k} x_j,\]
	where $(i,j,k)\in \mathrm{Cyc}(1,2,3)$
	induce graded $k$-algebra automorphisms of $S(\alpha_1,\alpha_2,\alpha_3)$.
\end{lemma}

The following lemma is straightforward.

\begin{lemma}\label{lem_del}
	Let the notation be as above.
	\begin{enumerate}
		\item The images of $\Omega_1$ under the maps $\delta_i$ are given by
		\begin{align*}
			\delta_1(\Omega_1) &= x_0^2 + \alpha_2\alpha_3 x_1^2 - \alpha_3 x_2^2 + \alpha_2 x_3^2, \\
			\delta_2(\Omega_1) &= x_0^2 + \alpha_3\alpha_1 x_2^2 - \alpha_1 x_3^2 + \alpha_3 x_1^2, \\
			\delta_3(\Omega_1) &= x_0^2 + \alpha_1\alpha_2 x_3^2 - \alpha_2 x_1^2 + \alpha_1 x_2^2.
		\end{align*}

		\item Let $0\neq f = \b_0x_0^2+\b_1x_1^2+\b_2 x_2^2+\b_3x_3^2
		\in k\Omega_1 + k\Omega_2$.
		Then $\b_0\b_1\b_2\b_3=0$
		if and only if $f$ is
\begin{align*}
\Omega _2 & = (1+\a_3)x_1^2+(1+\a_1\a_3)x_2^2+(1-\a_1)x_3^2, \\
\delta_1(\Omega_2) & = -(1+\a_3)x_0^2-(\a_1\a_3-\a_3)x_2^2+(\a_1+\a_3)x_3^2, \\
\delta_2(\Omega_2) & = -(1+\a_1\a_3)x_0^2-(\a_3-\a_1\a_3)x_1^2+(\a_1+\a_1\a_3)x_3^2,\\
\delta_3(\Omega_2) & = -(1-\a_1)x_0^2-(\a_1+\a_3) x_1^2-(\a_1+\a_1\a_3)x_2^2
\end{align*}
up to scalar.  It follows that at most one of $\b_i$ is zero.
	\end{enumerate}
\end{lemma}

\begin{proof}
The assertions follow from direct calculations. Since $\a_1+\a_3=-\a_2(1+\a_1\a_3)$ and $\a_2\neq 0$, $\a_1+\a_3=0$ if and only if $\a_1\a_3=-1$ if and only if $(\a_1, \a_3)=(1, -1)$ or $(-1, 1)$, which never occurs, so the last claim follows.
\end{proof}

In \cite{SV}, the authors assert that $S/(f)$ is a domain for every $f\in k\Omega_1+k\Omega_2$.
For the sake of completeness, we include a proof that $f$ is irreducible.

\begin{proposition}\label{prop_cenirr}
	Let $S=S(\alpha_1,\alpha_2,\alpha_3)$ be a $4$-dimensional Sklyanin algebra. Then every $ 0\neq f
	\in k\Omega_1 + k\Omega_2 $ is an irreducible central element of $S$.
\end{proposition}

\begin{proof}
	Suppose that $f$ is reducible so that $f=uv$ for some $u, v \in S_1$.  Since $S$ is a domain and $f$ is central, $uvu=fu=uf=u^2v$ implies $vu=uv$. If we write
	$
	u=\sum_{i=0}^{3} a_i x_i,
	v=\sum_{i=0}^{3} b_i x_i
	$, then
	\begin{align*}
		0 &=uv-vu= \sum_{0\leq i<j\leq 3}
		(a_ib_j-a_jb_i)(x_ix_j-x_jx_i) \\
		& =  \sum_{(i,j,k)\in \mathrm{Cyc}(1,2,3)}(\alpha_i(a_0b_i-a_ib_0)(x_jx_k + x_kx_j) + (a_jb_k-a_kb_j)(x_jx_k-x_kx_j)).
	\end{align*}
	Since the six elements $\{x_jx_k\pm x_kx_j \mid 1\leq j< k\leq 3\}$ are linearly independent, $a_0b_i-b_0a_i=0$ and $a_jb_k-a_kb_j=0$ for every $(i,j,k)\in \mathrm{Cyc}(1,2,3)$, so $u=\l v$ for some nonzero $\l \in k$.

	Since
	$f$ does not contain a term $x_ix_j$ for $i\neq j$,
	the equality
	\begin{align*}
		f &= \l (a_0x_0+a_1x_1+a_2x_2+a_3x_3)^2\\
		& =  \sum_{i=0}^{3}\l a_i^2x_i^2+\sum_{(i,j,k)\in \mathrm{Cyc}(1,2,3)}\l (a_0a_i(x_jx_k - x_kx_j) + a_ja_k(x_jx_k+x_kx_j))
	\end{align*} forces $a_ia_j=0$ for all $0\leq i\neq j\leq 3$. Thus, $f=\l x_i^2$ for some $i$, which contradicts the fact that at most one of the coefficients of $x_0^2, x_1^2, x_2^2, x_3^2$ in $f$ vanishes by Lemma \ref{lem_del} (2), so $f$ is irreducible.
\end{proof}

\begin{lemma}\label{lem_singular}
	Let $S=S(\alpha_1,\alpha_2,\alpha_3)$ be a $4$-dimensional Sklyanin algebra and $f = \b_0x_0^2 + \b_1 x_1^2 + \b_2 x_2^2 + \b_3 x_3^2 \in S_2$ with $\b_0\b_1\b_2\b_3\neq 0$. Then there exists $\Phi\in M_2(S_1)$ such that $\Phi^2 = f E_2$ over $S$ if and only if $f$ is
	\begin{align*}
			\Omega_1 & = x_0^2 - x_1^2 - x_2^2 - x_3^2, \\
			\delta_1(\Omega_1) &= x_0^2 + \alpha_2\alpha_3 x_1^2 - \alpha_3 x_2^2 + \alpha_2 x_3^2, \\
			\delta_2(\Omega_1) &= x_0^2 + \alpha_3\alpha_1 x_2^2 - \alpha_1 x_3^2 + \alpha_3 x_1^2, \\
			\delta_3(\Omega_1) &= x_0^2 + \alpha_1\alpha_2 x_3^2 - \alpha_2 x_1^2 + \alpha_1 x_2^2
		\end{align*}
	up to scalar.
	Moreover, if $f$ is in the above list,  then the matrix $\Phi$ such that $\Phi^2=fE_2$ is uniquely determined by $f$ up to conjugation and multiplication by $-1$.
\end{lemma}

\begin{proof}
	Without loss of generality, we may assume that $\b_0=1$.  Let $\Phi=\sum_{i=0}^{3}M_ix_i$, where $M_i\in M_2(k)$. If $\Phi^2=fE_2$, then $M_0^2=E_2$, $M_i^2=\b _iE_2, i=1,2,3$ and $\sum_{\substack{i,j=0\\ i\neq j}}^3 M_i M_j x_i x_j = 0$. Since
	$$M_iM_jx_ix_j+M_jM_ix_jx_i = \frac{1}{2}(M_iM_j+M_jM_i)(x_ix_j+x_jx_i) +  \frac{1}{2}(M_iM_j-M_jM_i)(x_ix_j-x_jx_i),$$
	we have
	\begin{align*}
		&\sum_{\substack{i,j=0\\ i\neq j}}^3 2M_i M_j x_i x_j \\
		=&\sum_{(i,j,k)\in \mathrm{Cyc}(1,2,3)}\{\alpha_i(M_0M_i-M_iM_0) + (M_jM_k+M_kM_j) \}(x_jx_k + x_kx_j)\\
		+& \sum_{(i,j,k)\in \mathrm{Cyc}(1,2,3)}\{(M_0M_i+M_iM_0) + (M_jM_k-M_kM_j)\}(x_jx_k - x_kx_j).
	\end{align*}
	It follows that
	\begin{align}
		&\alpha_i(M_0M_i-M_iM_0) + (M_jM_k+M_kM_j) = 0 \label{r_1},\\
		&(M_0M_i+M_iM_0) + (M_jM_k-M_kM_j) = 0\label{r_2}
	\end{align}
	for $(i,j,k)\in \mathrm{Cyc}(1,2,3)$.

(Case 1)
	First suppose that none of $M_0,M_1,M_2,M_3$ is a scalar matrix. If $\Phi^2=fE_2$ for some $\Phi\in M_2(S_1)$, then $M_i^2=\b_iE_2$, so we may assume that
	\[
	M_0=
	\begin{pmatrix}
		1&0\\
		0&-1
	\end{pmatrix},
	\,
	M_i=
	\begin{pmatrix}
		a_i&b_i\\
		c_i&-a_i
	\end{pmatrix}
	\quad (i=1,2,3)
	\]
	 by simultaneous conjugation.  By (\ref{r_1}) and (\ref{r_2}), we obtain
	 \begin{align*}
& \alpha_1 b_1=\alpha_1 c_1=\alpha_2 b_2=\alpha_2 c_2=\alpha_3 b_3=\alpha_3 c_3=0, \\
	 & a_1=a_2=a_3=0,
	 \end{align*}
	 so $M_1=M_2=M_3=0$.
	 Since $\b_0\b_1\b_2\b_3\neq 0$,
	 this case cannot occur.

	(Case 2) Next suppose that $M_0$ is a scalar matrix so that $M_0=\e E_2$ where $\e=\pm 1$.
	If $\Phi^2=fE_2$ for some $\Phi\in M_2(S_1)$, then (\ref{r_1}) and (\ref{r_2}) reduce to
	\begin{align}
	& M_jM_k+M_kM_j=0, \\
	& M_jM_k=-\e M_i
	\end{align}
	for every
	$(i,j,k)\in \mathrm{Cyc}(1,2,3)$.
 Since
 $$\beta _iM_j=M_i^2M_j=-\e M_iM_k=\e M_kM_i=-\e^2M_j=-M_j,$$
 $\beta_i=-1$ for $i=1, 2, 3$, so
 $f=\Omega _1$.

	We now show that there exists (unique) $\Phi\in M_2(S_1)$ such that $\Phi^2=\Omega_1E_2$.
	Since
	\begin{align}
	& M_jM_k+M_kM_j=0, \; (1\leq j\neq k\leq 3),  \\
	& M_i^2=-E_2, \; (1\leq i\leq 3),
	\end{align}
	the assignment $\rho(u_i) = M_i$ for $1 \leq i \leq 3$ induces a $2$-dimensional (irreducible) representation $\rho: Cl_3(k) \longrightarrow M_2(k)$ of the Clifford algebra
	\[
	Cl_3(k) := \frac{k\langle u_1, u_2, u_3 \rangle}{(u_i u_j + u_j u_i + 2\delta_{ij})_{1 \le i,j \le 3}},
	\]
	  where $\delta_{ij}$ denotes the Kronecker delta.  By \cite[Theorem 5.8, Proposition 5.9]{LaM}, $Cl_3(k)$ has two irreducible representations and the two representations are related by $\e=\pm 1$.
Therefore, $(M_1,M_2,M_3)$ is simultaneously conjugate to
	\[
	\left(\e \begin{pmatrix}\sqrt{-1}&0\\0&-\sqrt{-1}\end{pmatrix},\,
	\e \begin{pmatrix}0&\sqrt{-1}\\\sqrt{-1}&0\end{pmatrix},\,
	\e \begin{pmatrix}0&1\\ -1&0\end{pmatrix}\right).
	\]
	In fact if we choose $(M_1,M_2,M_3)$ as above (and $M_0=\e E_2$), then we can verify that
	\begin{align*} \Phi^2&=\left(\sum_{i=0}^3M_ix_i\right)^2=\e^2\begin{pmatrix}
		x_0+\sqrt{-1}x_1 & \sqrt{-1}x_2+x_3\\
		\sqrt{-1}x_2-x_3 & x_0-\sqrt{-1}x_1
	\end{pmatrix}^2\\
&=(x_0^2-x_1^2-x_2^2-x_3^2)E_2=\Omega _1E_2,
	\end{align*}
	so there exists a matrix $\Phi \in M_2(S_1)$ such that $\Phi^2=\Omega _1E_2$,
	and it is unique up to conjugation and multiplication by $-1$.

	(Case 3) Next suppose that $M_1$ is a scalar matrix.
	Since
	$$\Psi:=\delta_1^{-1}(\Phi)=\dfrac{1}{\sqrt {\a_2\a_3}}M_1x_0+\sqrt{-1}M_0x_1-\dfrac{1}{\sqrt {-\a_2}}M_3x_2-\dfrac{1}{\sqrt{\a_2}}M_2x_3
=:\sum_{i=0}^{3}N_ix_i,$$
	$N_0$ is a scalar matrix.  Since $\Psi^2=(\delta_1^{-1}(\Phi))^2=\delta_1^{-1}(\Phi^2)=\delta_1^{-1}(f)E_2$,
it follows from  the  argument of (Case 2) that
$\delta_1^{-1}(f)$ is $\Omega _1$ up to scalar, so $f$ is $\delta_1(\Omega_1)$ up to scalar, and we can reduce to (Case 2).

(Case 4) We can reduce the case where $M_2$ is a scalar matrix or $M_3$ is a scalar matrix to (Case 2) by a similar argument in (Case 3).
\end{proof}

In the proof above, we have constructed a noncommutative graded matrix factorization of $\Omega_1$ over $S(\alpha_1,\alpha_2, \alpha_3)$.  In the proof below, we construct a noncommutative graded matrix factorization of $\Omega_2$ over $S(\alpha_1,\alpha_2, \alpha_3)$.

\begin{lemma}	\label{lem_3variables} $S(\a_1, \a_2, \a_3)/(\Omega _2)$ is standard smooth.
\end{lemma}

\begin{proof}
	Recall that $\Omega_2 = (1+\alpha_3)x_1^2
	+ (1+\alpha_1\alpha_3)x_2^2
	+ (1-\alpha_1)x_3^2 $.
	If we define
	\begin{align*}
	&\Phi=
	\begin{pmatrix}
		-2\sqrt{-1}\,\dfrac{\c_1}{\c_2}x_0+\dfrac{1-\alpha_1}{\c_2}x_3
		&
		(1+\alpha_3)x_1-\sqrt{-1}\,\dfrac{1+\alpha_1\alpha_3}{\c_1}x_2
		\\[8pt]
		(1+\alpha_3)x_1+\sqrt{-1}\,\dfrac{1+\alpha_1\alpha_3}{\c_1}x_2
		&
		-2\sqrt{-1}\,\dfrac{\c_1}{\c_2}x_0-\dfrac{1-\alpha_1}{\c_2}x_3
	\end{pmatrix},\\
	&\Psi=
	\begin{pmatrix}
		\c_2x_3 & x_1-\sqrt{-1} \c_1x_2\\[4pt]
		x_1+\sqrt{-1} \c_1x_2 & -\c_2x_3
	\end{pmatrix}, \text{ where }	\c_1^2=\frac{1+\alpha_1\alpha_3}{1-\alpha_3},\ \c_2^2=\frac{1+\alpha_1}{1-\alpha_3},
	\end{align*}
then we can check that $\Phi\Psi=\Psi\Phi=\Omega _2E_2$ by brute-force computations so that $(\Phi, \Psi), (\Psi, \Phi)\in \LMF^2_{S(\a_1, \a_2, \a_3)}(\Omega_2)$. (Here we use identities such as
\begin{align*}
& 2x_0x_1=(1+\alpha_1)x_2x_3-(1-\alpha_1)x_3x_2, \\
& 2x_0x_2=(1+\alpha_2)x_3x_1-(1-\alpha_2)x_1x_3, \\
& 2x_0x_3=(1+\alpha_3)x_1x_2-(1-\alpha_3)x_2x_1,
\end{align*}
and $\alpha_1+\alpha_2+\alpha_3+\alpha_1\alpha_2\alpha_3=0$.)
	Since $\dim \Phi=4\neq \dim \Psi=3$, $\mathrm{Coker}\, \Phi \not \cong \mathrm{Coker}\, \Psi $ by Lemma \ref{lem.lisc1}.  It follows that $(\Phi, \Psi)\not \cong (\Psi, \Phi)$, so $S(\a_1, \a_2, \a_3)/(\Omega_2)$ is standard smooth by Lemma \ref{lem.isns}.
\end{proof}

\begin{remark} \label{rem.det} For $\Phi, \Psi$ in the proof of the above lemma, $\cV(\det \Phi^\natural)\neq \cV(\det\Psi^\natural)$.  In fact, we can check that $\cV(\det \Phi^\natural)\subset \PP^3$ is smooth while $\cV(\det\Psi^\natural)\subset \PP^3$ is singular.
\end{remark}

By Lemma \ref{lem_3variables}, we obtain the following result.

\begin{proposition}\label{prop.Om2}
Let $S=S(\alpha_1,\alpha_2,\alpha_3)$ be a $4$-dimensional Sklyanin algebra and $0\neq  f = \b_0x_0^2 + \b_1 x_1^2 + \b_2 x_2^2 + \b_3 x_3^2\in k\Omega_1 + k\Omega_2 $.
If $\b_0\b_1\b_2\b_3=0$, then $S/(f)\cong S/(\Omega_2)$.  In this case, $S/(f)$  is standard smooth.
\end{proposition}

\begin{proof}
By Lemma \ref{lem_del}, $\b_0 \b_1 \b_2 \b_3 = 0$ if and only if $f = \delta_i(\Omega_2)$ for some $0 \leq i \leq 3$, where $\delta_0 := \operatorname{id}_S$. Since each $\delta_i$ is a graded algebra automorphism of $S$, we have an isomorphism of graded algebras $S/(f) \cong S/(\Omega_2)$, which is standard smooth by Lemma \ref{lem_3variables}.
\end{proof}

We finally give a characterization of the smoothness of $S/(f)$ for a central Sklyanin quadric. According to Theorem \ref{thm.10.2}, there are four singular quadrics $S/(f)$ for each $S$.  We identify such four singular quadrics and show that they are isomorphic to each other, that is, there is a unique singular quadric $S/(f)$ up to  isomorphism for each fixed $S$.  By this characterization, we will show that all smooth $S/(f)$ are standard.

\begin{theorem}\label{thm_sing}
Let $S=S(\alpha_1,\alpha_2,\alpha_3)$ be a $4$-dimensional Sklyanin algebra and $0\neq f = \b_0x_0^2 + \b_1 x_1^2 + \b_2 x_2^2 + \b_3 x_3^2\in k\Omega_1 + k\Omega_2 $.
\begin{enumerate}
\item The following conditions are equivalent.
\begin{enumerate}
\item $S/(f)$ is singular.
\item $f$ is one of $\Omega_1$, $\d_1(\Omega_1)$, $\d_2(\Omega_1)$,
or $\d_3(\Omega_1)$ up to scalar.
\item $S/(f) \cong S/(\Omega_1)$.
\end{enumerate}
\item  If $S/(f)$ is smooth, then $S/(f)$ is standard.
\end{enumerate}
\end{theorem}

\begin{proof}
(1) (a) $\Rightarrow$ (b):
If $S/(f)$ is singular, then there exists $(\Phi, \Phi)\in \LMF_S^2(f)$ by Lemma \ref{lem.isns}(1), and
$\b_0\b_1\b_2\b_3\neq 0$ by Proposition \ref{prop.Om2}, so (b)
holds from  Lemma~\ref{lem_singular}.

(b) $\Rightarrow$ (c): Since each $\delta_i$ is a graded algebra
automorphism of $S$, we have an isomorphism of graded algebras
$S/(f) \cong S/(\Omega_1)$ in each case.

(c) $\Rightarrow$ (a): By Lemmas \ref{lem_singular} and \ref{lem.isns},
we see that $S/(f) \cong S/(\Omega_1)$ is singular.

\smallskip
(2) If  $S/(f)$ is smooth but non-standard, then there exists $(\Phi, \Phi)\in \LMF_S^2(f)$  by Lemma \ref{lem.isns}(3), and $\b_0\b_1\b_2\b_3\neq 0$ by Proposition \ref{prop.Om2}, so (b)
holds by Lemma~\ref{lem_singular} again.  However, in this case, $S/(f)$ is singular by (1), which is a contradiction.
Therefore, if $S/(f)$ is smooth, then it is standard.
\end{proof}

\end{document}